\documentclass[11pt,letterpaper]{amsart}
\usepackage{lmodern}
\usepackage[T1]{fontenc}
\usepackage[english]{babel}
\usepackage{amsmath,amscd,amssymb,amsthm,amsxtra}
\usepackage{xfrac}
\usepackage{microtype}
\usepackage{eucal}
\usepackage{mathrsfs}
\usepackage{mathtools}
\usepackage{accents}
\usepackage{enumerate}
\usepackage{color,graphicx,esint}
\usepackage{epsfig,epstopdf}
\usepackage{tikz-cd}
\usetikzlibrary{decorations.pathreplacing,patterns}
\usepackage[notref,notcite]{}
\usepackage{hyperref}
\hypersetup{colorlinks={true}, linkcolor={black}, citecolor={black}}

\theoremstyle{plain}
\newtheorem{theorem}{Theorem}[section]
\newtheorem{corollary}[theorem]{Corollary}
\newtheorem{proposition}[theorem]{Proposition}
\newtheorem{lemma}[theorem]{Lemma}

\theoremstyle{definition}
\newtheorem{definition}[theorem]{Definition}
\newtheorem{remark}[theorem]{Remark}

\theoremstyle{remark}
\numberwithin{theorem}{section}
\numberwithin{equation}{section}
\numberwithin{figure}{section}

\def\R{\mathbb{R}}

\def\N{\mathbb{N}}
\def\1{{\bf 1}}
\def\e{\mathbf{e}}
\def\d{\mathrm{d}}

\def\XXint#1#2#3{{\setbox0=\hbox{$#1{#2#3}{\int}$}
        \vcenter{\hbox{$#2#3$}}\kern-.5\wd0}}

\DeclareMathOperator{\cof}{cof}
\DeclareMathOperator{\conv}{conv}

\DeclareMathOperator{\diag}{diag}
\DeclareMathOperator{\diam}{diam}
\DeclareMathOperator{\dist}{dist}

\DeclareMathOperator{\interior}{int}
\DeclareMathOperator{\Id}{Id}

\DeclareMathOperator{\osc}{osc}

\DeclareMathOperator{\Span}{span}
\DeclareMathOperator{\spt}{spt}
\DeclareMathOperator{\trace}{tr}

\newcommand{\mres}{\mathop{\hbox{\vrule height 7pt width .5pt depth 0pt
\vrule height .5pt width 6pt depth 0pt}}\nolimits}

\begin{document}

\title[Optimal Transports
between Degenerate Densities]{On the Regularity of Optimal Transports \\
between Degenerate Densities}

\author[Y.\ Jhaveri]{Yash Jhaveri}
\address{Columbia University, Department of Mathematics, 2990 Broadway, New York, New York 10027, USA}
\email{yjhaveri@math.columbia.edu}
\author[O.\ Savin]{Ovidiu Savin}
\address{Columbia University, Department of Mathematics, 2990 Broadway, New York, New York 10027, USA}
\email{savin@math.columbia.edu}
\thanks{YJ was supporting in part by NSF grant DMS-1954363. OS was supported in part by NSF grants DMS-1800645 and DMS-2055617.}

%~~~ABSTRACT~~~%
\begin{abstract}
We study the most common image and informal description of the optimal transport problem for quadratic cost, also known as the second boundary value problem for the Monge--Amp\`{e}re equation---What is the most efficient way to fill a hole with a given pile of sand?---by proving regularity results for optimal transports between degenerate densities.
In particular, our work contains an analysis of the setting in which holes and sandpiles are represented by absolutely continuous measures concentrated on bounded convex domains whose densities behave like nonnegative powers of the distance functions to the boundaries of these domains.
\end{abstract}
\maketitle
%MSC 2010 Primary: 35J96, MSC 2010 Secondary: 35J70

%~~~MATH~~~%

\section{Introduction}

The optimal transport problem, formulated by Gaspard Monge in 1781, asks whether or not it is possible to find a map minimizing the total cost of moving a distribution of mass $f$ to another $g$ given that the cost of moving from $x$ to $y$ is measured by $c = c(x,y)$. 
Since its inception, optimal transportation has drawn together and impacted many areas of mathematics: fluid mechanics, functional analysis, geometry, general relativity, and probability, just to name a few (see, e.g., \cite{BB,FMP,LV,M2,V2}).
The most fundamental case is that of the quadratic cost on $\R^n$, when $c(x,y) = |x-y|^2$ for $x,y \in \R^n$.
It is the model for all sufficiently smooth cost functions on all sufficiently smooth (Riemannian) geometries (\cite{DPF}), and it is at the core of many applications (\cite{V2}).
Precisely, it is 
\[
\min_T \bigg{\{} \int_{\R^n} |x-T(x)|^2 \, \d f(x) : T_{\#} f = g \bigg{\}}.
\]

Under certain conditions on the nonnegative measures $f$ and $g$, Brenier discovered that the optimal transport problem for the quadratic cost on $\R^n$ is uniquely solvable $f$-almost everywhere (\cite{B}; see also \cite{M1}).
Moreover, he characterized minimizing maps as gradient maps of convex potentials: $T_{min} = \nabla u$ for some convex function $u : \R^n \to \R$.
When $f$ and $g$ are absolutely continuous with respect to Lebesgue measure, he also established that any convex potential $u$ defining $T_{min}$ satisfies a Monge--Amp\`{e}re equation,
\[
g(\nabla u)\det D^2 u = f
\text{ and }
\nabla u(\spt f) = \spt g,
\]
in a suitable weak sense (the Brenier sense; see Lemma~\ref{lem: Brenier soln}), where $f = f \, \d x$ and $g = g \, \d y$.
(In this work, we equate absolutely continuous measures with their densities.
It will either be clear from the context or explicitly stated when absolute continuity is assumed.)
In turn, Brenier linked the optimal transport problem and the second boundary value problem for the Monge--Amp\`{e}re equation: given two convex domains and a nonnegative function on their product, find a convex function whose gradient maps one domain onto the other with Jacobian determinant proportional to the given function.

Unfortunately, optimal transports can behave rather poorly.
Indeed, Caffarelli observed that $T_{min} = \nabla u$ can be discontinuous under the seemingly ideal conditions that $f$ and $g$ are the characteristic functions of smooth, bounded domains of equal volume (\cite{C1}; see also \cite{J}).
In principle then, a convex potential of an optimal transport (on the support of the source measure) even between ``nice'' measures is no better than an arbitrary convex function.
That said, in this same work, Caffarelli showed that the optimal transport must be locally H\"{o}lder continuous in $X := \interior(\spt(f))$ under a geometric condition\,---\,$Y := \interior(\spt(g))$ is convex\,---\,and a uniform ellipticity type condition\,---\,the Monge--Amp\`{e}re measure associated to $u$ is doubling in $\spt(f)$.
In subsequent works, Caffarelli established the global H\"{o}lder continuity of $\nabla u$ assuming that both $X$ and $Y$ are convex (\cite{C2}), and the global H\"{o}lder continuity of $D^2 u$, the Hessian of $u$, additionally assuming $X$ and $Y$ are $C^2$ and uniformly convex and $f$ and $g$ are positive and H\"{o}lder continuous in $\overline{X}$ and $\overline{Y}$ respectively (\cite{C3}).

In this nondegenerate setting, Urbas also proved that $D^2 u$ is H\"{o}lder continuous up to $\partial X$ when $X$ and $Y$ are uniformly convex, but under a $C^3$ regularity assumption on $X$ and $Y$ (\cite{U}).
More recently, Chen, Liu, and Wang demonstrated that these domain regularity assumptions can be weakened to $C^{1,1}$ in $n \geq 3$ dimensions and $C^{1,\alpha}$ in two dimensions (\cite{CLW1,CLW2}).
In two dimensions and at the same time as Caffarelli, Delano\"{e} established the existence of globally smooth solutions to the second boundary value problem for the Monge--Amp\`{e}re equation given smooth data (\cite{D}).

In the degenerate setting of arbitrary open, bounded source and target domains, but still considering densities bounded away from zero and infinity, Figalli (\cite{F}), Figalli--Kim (\cite{FK}), Goldman--Otto (\cite{GO}), and Goldman (\cite{G}) showed that the closure of the discontinuity set of an optimal transport, also known as the singular set, has zero measure in $\overline{X}$.

In this paper, we consider a different degenerate setting, one in which $f$ and $g$ are permitted to vanish at times, e.g., continuously at the boundaries of $X$ and $Y$.
This scenario encompasses a study of the most common image and informal description of the optimal transport problem:
\[
\text{\it What is the most efficient way to fill a hole with a given pile of sand?}
\]

Our first result is a global H\"{o}lder continuity regularity result for optimal transports between absolutely continuous, doubling measures (see Section~\ref{sec: prelim} for the definition of a doubling measure in this context).
This doubling assumption is different from Caffarelli's doubling assumption in that it is only on the data of the problem rather than on the data and the solution, as it is in Caffarelli's case.

\begin{theorem}
\label{thm: global Csigma}
Let $X$ and $Y$ be open, bounded convex sets in $\R^n$, and suppose that $f$ and $g$ are densities which define doubling measures concentrated on $X$ and $Y$ respectively.
Let $T_{min}$ be the optimal transport taking $f$ to $g$.
Then $T_{min} \in C^{\sigma}(\overline{X})$, for some $\sigma \in (0,1)$, depending on $n$, the doubling constants of $f$ and $g$, and the inner and outer diameters of $X$ and $Y$.
\end{theorem}

Our second result establishes (optimal) global regularity for the optimal transport in the plane when $f$ and $g$ are comparable to nonnegative powers of the distance functions to the boundaries of their supports, which we assume are convex:
\[
f \sim d_{\partial X}^\alpha \text{ for some } \alpha \geq 0 \text{ and } g \sim d_{\partial Y}^\beta \text{ for some } \beta \geq 0.
\]  
In this work, $d_{\partial \ast}$ represents the distance function to the boundary of $\ast$; $d_{\partial \ast} > 0$ in $\ast$ and $d_{\partial \ast} = 0$ outside of $\ast$.
Thus, we assume our sandpile and hole (turned upside down) have precise shapes at their boundaries.

Here we show that the optimal transport effectively splits along the tangential and normal directions to $\partial X$.
Let $u' : [0,1] \to [0,1]$ be the optimal transport taking a density which behaves like $x^\alpha$ near $0$ to another density which behaves like $y^\beta$ near $0$.
Then, by the mass balance formula,
\[
u'(t) \sim t^{\gamma} \text{ with } \gamma := \frac{1+\alpha}{1+\beta}.
\]
In other words, informally, we find that $T_{min}$ behaves like the identity map $t$ moving along the boundary of $X$ and the one dimensional transport $t^{\gamma}$ moving orthogonally in from the boundary of $X$.

In order to precisely state our theorem and expansion, we must define three H\"{o}lder exponents, $\lambda$, $\mu$, and $\omega$, to formalize what we mean by $\sim$ above.
We state our theorem assuming that $\alpha > 0$ and $\beta > 0$, and make a remark after to address the mild difference when either $\alpha = 0$ or $\beta = 0$.
There are two cases to consider.
When $\alpha \geq \beta$, let
\[
\text{$\mu := \lambda \frac{1+\gamma}{2}$ and $\omega := \lambda$, for any fixed $0 < \lambda \leq \min \bigg\{ \alpha \frac{2}{1+\gamma}, \frac{2}{1+\gamma}, \beta \bigg\}$.}
\]
If $\alpha = \beta$, i.e., $\gamma = 1$, we additionally assume that $\lambda < 1$.
On the other hand, when $\alpha < \beta$, set
\[
\text{$\mu := \lambda $ and $\omega := \lambda \frac{1+\gamma}{2\gamma}$, for any fixed $0 < \lambda \leq \min \bigg\{ \alpha, \frac{2\gamma}{1+\gamma}, \beta \frac{2\gamma}{1+\gamma}\bigg \}$.}
\]

\begin{theorem}
\label{thm: global C2 in plane}
Let $X$ and $Y$ be open, bounded, and $C^{1,1}$ uniformly convex sets in $\R^2$. 
Suppose $\alpha$ and $\beta$ are two positive constants. 
Let $a \in C^{\mu}(\overline{X})$ and $b \in C^{\omega}(\overline{Y})$ be two positive functions.
Suppose that $T_{min}$ is the optimal transport taking $f = ad_{\partial X}^\alpha$ to $g = bd_{\partial Y}^\beta$.
If $\gamma \geq 1$, then $T_{min} \in C^{1 + \lambda}(\overline{X})$.
On the other hand, if $\gamma < 1$, then $T_{min} \in C^{\gamma(1+ \omega)}(\overline{X})$.
\end{theorem}

At the heart of Theorem~\ref{thm: global C2 in plane} is an expansion of a convex potential defining $T_{min}$ at points on $\partial X$.
In particular, up to a translation placing $0 \in \partial X$ and subtracting off the tangent plane to $u$ at $0$, we prove that
\[
|u(Mx) - p_1x_1^2 - p_2x_2^{1+\gamma}| \leq C(|x_1|^2 + x_2^{1+\gamma})^{1 + \frac{\lambda}{2}} \text{ in } \{ x_2 \geq 0 \} \cap M^{-1}\overline{X},
\]
where $M$ is some linear transformation and $p_1$, $p_2$, and $C$ are three positive constants.

\begin{remark}
The case $\alpha = 0 = \beta$ is, by now, well-understood.
When $\alpha = 0$ (respectively $\beta = 0$), the upper bound on $\lambda$ becomes independent of any dependence on $\alpha$ (respectively $\beta$).
\end{remark}

Our final main result is a Liouville theorem in the flat setting, when $X = \{ x_n > 0 \}$ and $Y = \{ y_n > 0 \}$, with $a \equiv 1 \equiv b$.
Here $u_n \geq 0 \text{ in } \{ x_n > 0 \}$, and our Monge--Amp\`{e}re equation is
\begin{equation}
\label{eqn: MA Liouville}
\det D^2 u = \frac{x_n^\alpha}{u_n^\beta} \text{ in } \{ x_n > 0 \} \text{ and } u_n = 0 \text{ on } \{ x_n = 0 \}.
\end{equation}
We remark that this equation is invariant under affine transformations that keep the tangential variables $x' = (x_1,\dots,x_{n-1})$ separate from the normal variable $x_n$: $Ax = (A'x', a_n x_n)$.
Furthermore, since all three notions of weak solution to the Monge--Amp\`{e}re equation (Alexandrov, Brenier, and viscosity) are equivalent in this case, the following theorem classifies not only Brenier solutions to \eqref{eqn: MA Liouville}, but also Alexandrov and viscosity solutions to \eqref{eqn: MA Liouville}.

\begin{theorem}
\label{thm: liouville}
Let $u$ be convex and such that $(\nabla u)_{\#} d_{\partial \{x_n > 0\}}^\alpha =  d_{\partial \{y_n > 0\}}^\beta$, for two given constants $\alpha \geq 0$ and $\beta \geq 0$.
Then
\[
u(x) = p_0 + p' \cdot x' + P' x' \cdot x' + p_n x_n^{1+\gamma}
\]
for some $p_0 \in \R$, $p' \in \R^{n-1}$, positive definite matrix $P'$, and constant $p_n > 0$.
\end{theorem}

This paper is organized as follows.
The next section collects some facts from measure theory and convex analysis.
In Section~\ref{sec: Caff Bdry Reg Revisit}, we revisit Caffarelli's boundary regularity theory, and prove Theorem~\ref{thm: global Csigma}.
In Section~\ref{sec: The Flat Case}, we study the flat setting, and we prove our (Liouville) Theorem~\ref{thm: liouville}.
In Section~\ref{sec: Flat Implies Smooth}, we establish a pointwise ``flat implies smooth'' result.
Finally, in Section~\ref{sec: 2D}, we prove Theorem~\ref{thm: global C2 in plane}.

%~~~~~~~~~~~~~~~~~~~~~~~~~~~~~~~~~~~~~~~~~~~~~~~~~~~~~~~~~~~~%
\section{Preliminaries}
\label{sec: prelim}

Throughout this work, $c$ and $C$ will denote positive constants that may change from line to line.
It will be clear from the context, if any change occurs.
Sometimes some of the quantities on which $c$ and $C$ depend will be explicit and denoted in parentheses or as subscripts; other times, especially when these quantities are contextually clear, these quantities will be implicit.

Let us start with a pair of definitions and an important lemma by John.

\begin{definition}
We say that a map $T$ {\it pushes-forward} a measure $f$ to another measure $g$, $T_{\#}f = g$, if
\[
\int \varphi \circ T \, \d f = \int \varphi \, \d g \text{ for all $\varphi$ Borel and bounded.}
\]
\end{definition}

\begin{definition}
A nonnegative measure $f$ is {\it doubling} (on bounded convex domains) if there is a constant $C \geq 1$ such that the following holds: given an open, bounded convex set $S$ whose barycenter is contained in $\spt f$, 
\[
f(S) \leq C f(\tfrac{1}{2}S),
\]
where $\frac{1}{2}S$ is the dilation of $S$ with respect to its center of mass by $1/2$.
\end{definition}

\begin{definition}
An {\it ellipsoid} is the image under a symmetric positive definite affine transformation of $B_1(0)$.
In particular, let $E$ be any symmetric positive definite matrix and $x \in \R^n$, the ellipsoid generated by $E$ and centered at $x \in \R^n$ is 
\[
\mathcal{E}_{E,x} := x + E(B_1(0)).
\]
Given $r > 0$, we let 
\[
r\mathcal{E} = r\mathcal{E}_{E,x} := \mathcal{E}_{rE,x}
\]
be the dilation of $\mathcal{E}$ with respect to its center by $r$.
(Given an ellipsoid $\mathcal{E}$, we can assume its generating matrix $E$ can be diagonalized with a determinant $1$ orthogonal matrix.)
\end{definition}

\begin{lemma}[John's Lemma]
Let $S \subset \R^n$ be a bounded convex set with nonempty interior and center of mass $z$.
A unique ellipsoid $\mathcal{E}$ also with center of mass $z$ exists such that 
\[
\mathcal{E} \subset S \subset n^\frac{3}{2}\mathcal{E}.
\]
\end{lemma}

With these definitions and John's lemma in hand (see, e.g., \cite{Gu} for a proof), we prove that a measure that is doubling on ellipsoids is doubling.

\begin{corollary}
\label{cor: doubling on ellipsoids implies doubling}
Let $f$ be a nonnegative measure.
If $f$ is doubling on ellipsoids, then $f$ is doubling.
\end{corollary}

\begin{proof}
Let $S$ be an open, bounded convex set.
Then, by John's lemma, $S$ is comparable to an ellipsoid $\mathcal{E} $: $\mathcal{E}  \subset S \subset n^{3/2}\mathcal{E}$, and the center of mass of $\mathcal{E} $ is the same as the center of mass of $S$.
If $C \geq 1$ is the doubling constant for $f$ on ellipsoids and $k \geq 1$ is chosen such that $n^{3/2}/2^k \leq 1/2$, then
\[
f(S) \leq f( n^{\frac{3}{2}}\mathcal{E} ) \leq C f(\tfrac{1}{2}n^{\frac{3}{2}} \mathcal{E} ) \leq \dots \leq C^k f(\tfrac{1}{2^k}n^{\frac{3}{2}} \mathcal{E} ) \leq C^k f(\tfrac{1}{2}\mathcal{E} ) \leq C^k f( \tfrac{1}{2}S ).
\]
\end{proof}

As a consequence of Lemma~\ref{cor: doubling on ellipsoids implies doubling}, we can show that measures that are comparable to the distance function to the boundary of a convex domain are doubling.

\begin{lemma}
\label{lem: doubling on ellipsoids}
Let $X \subset \R^n$ be an open, bounded convex set.
The density $f = a d^\alpha_{\partial X}$ defines a doubling measure on ellipsoids if $0 < \inf_X a, \sup_X a < \infty$.
In particular, a constant $C \geq 1$ exists for which
\[
\int_{\mathcal{E}}  f  \leq C \int_{\frac{1}{2}\mathcal{E} } f
\]
given any ellipsoid $\mathcal{E} $ centered in $\overline{X}$.
\end{lemma}

\begin{proof}
There are two cases to consider.

\medskip
{\bf Case 1: $\mathcal{E}  \subset X$.}
Let $d := d_{\partial X}(z)$, with $z$ taken to be the center of $\mathcal{E} $.
Up to a translation and rotation, $\mathcal{E}, X \subset \{x_n > 0 \}$, and the origin is the closet point on $\partial X$ to $z$.

First, note $d_{\partial X}(x) \leq 2d$ for all $x \in \partial \mathcal{E}$, and so, for all $x \in \mathcal{E}$.
Indeed, if $x_n \leq z_n$, then there is nothing to show.
If $x_n > z_n$, then
\[
d_{\partial X}(x) \leq d_{\partial {\{x_n > 0 \}}}(x) = x_n = (x_n - z_n) + z_n = (z_n - x^\ast_n) + z_n \leq 2z_n,
\]
since $x^\ast_n, x_n > 0$.
Here $x^\ast \in \partial \mathcal{E}$ is the dual point to $x$.
So
\[
\int_{\mathcal{E}} f \leq 2^\alpha (\sup a) d^\alpha|\mathcal{E}|.
\]

Second, for $i = 1, \dots, n$, let $r_i$ and $\e_i$ be the principle radii and directions of $\mathcal{E}$.
Hence, $\triangle := \conv \{ z \pm r_i\e_i : i = 1, \dots, n \} \subset \mathcal{E} $.
Also, $|\triangle|/|\mathcal{E} | \geq c(n)$.
Now consider $\triangle_d := \conv \{ z \pm \max\{r_i,d\}\e_i : i = 1, \dots, n \}$, which contains $\triangle$ and is contained in $X$.
For all $x \in \frac{1}{2}\triangle_d$ then, $\dist_{\partial \triangle_d}(x) \geq d/2^{\frac{3}{2}}$.
In turn,
\[
d_{\partial X}(x) \geq d_{\partial \triangle_d}(x) \geq \frac{d}{2^{\frac{3}{2}}} \text{ for all } x \in \tfrac{1}{2}\triangle.
\]
It follows that
\[
d^\alpha|\mathcal{E} | \leq \frac{2^{n+\frac{3\alpha}{2}}}{c(n) \inf a} (\inf a) \frac{d^\alpha}{2^{\frac{3\alpha}{2}}} |\tfrac{1}{2} \triangle|
\leq \frac{2^{n+\frac{3\alpha}{2}}}{c(n) \inf a} \int_{\frac{1}{2} \triangle} f 
\leq C(n,\alpha,\inf a)  \int_{\frac{1}{2} \mathcal{E} } f.
\] 

Finally, the above two inequalities together yield
\[
\int_\mathcal{E} f\leq C(n,\alpha,\inf a,\sup a)  \int_{\frac{1}{2} \mathcal{E}} f.
\]

\medskip
{\bf Case 2: $\mathcal{E}  \setminus X$ is nonempty.}
Up to a translation, we can assume that the center of $\mathcal{E} $ is the origin.
Let $z$ be the center of mass of $(\frac{1}{2}\mathcal{E} ) \cap X$, $d := d_{\partial X}(z)$, and suppose that the nearest point to $z$ on $\partial X$ lives on the plane $\{ x \cdot \e = 0 \}$ for some $|\e| = 1$.

Using that open, bounded convex sets are balanced with respect to their center of mass, $d_{\partial \{ x \cdot \e > 0 \}}$ is $1$-homogeneous, and arguing like we did to produce the first inequality above, we see that
\[
d_{\partial X}(x) \leq C(n) d \text{ for all } x \in \mathcal{E}  \cap X.
\]

Now let $\mathcal{J}$ be the John ellipsoid of $(\frac{1}{2}\mathcal{E} ) \cap X$, which also has center $z$.
Notice that $\mathcal{J} \subset X$.
So
\[
\int_{\mathcal{E}} f \leq C d^\alpha|\mathcal{E}  \cap X| 
\leq C d^\alpha 2^n |(\tfrac{1}{2} \mathcal{E} ) \cap X |
\leq C d^\alpha |\mathcal{J}| 
\leq C \int_{\frac{1}{2}\mathcal{J}} f 
\leq  C \int_{\frac{1}{2}\mathcal{E} } f. 
\]
Here we have used the arguments of Case 1 on $\mathcal{J}$ and that $\mathcal{E} \cap X \subset 2[(\frac{1}{2}\mathcal{E} ) \cap X]$, which uses that $0 \in \overline{X}$.
\end{proof}

\begin{remark}
By Corollary~\ref{cor: doubling on ellipsoids implies doubling}, if $0 < \inf_X a, \sup_X a < \infty$, then $f = a d^\alpha_{\partial X}$ defines a doubling measure.
\end{remark}

We conclude this section with three lemmas.
These lemmas use nothing about the optimal transport problem; they are facts about convex functions with centered sections.
The first lemma's proof may be found in \cite{C2}.

\begin{lemma}[Centered Sections]
Let $u : \R^n \to \R$ be a convex function whose graph contains no complete lines.
Then, for every $h > 0$ and $z \in \R^n$, there exists an affine function $\ell$ such that $\ell(z) = u(z) + h$ and the set
\[
S_h^c(z) := \{ x \in \R^n : u < \ell \}
\]
is centered at $z$.
\end{lemma}

\begin{lemma}
\label{lem: centered sections are balanced}
Let $S_h^c(z)$ be a centered section for $u$ at $z \in \R^n$.
Let $z_1$ and $z_2$ be two opposite points on $\partial S_h^c(z)$, i.e., $z_2 = z + l(z - z_1)$ for some $l > 0$.
Then
\[
n^{-\frac{3}{2}} \leq l \leq n^{\frac{3}{2}}.
\]
\end{lemma}

\begin{proof}
Without loss of generality, we may assume that $z = 0$ and that $z_1$ and $z_2$ lie on the $\e_1$-axis.
So, by John's lemma, $\mathcal{E} \subset S_h^c(0) \subset n^{3/2}\mathcal{E}$, for some ellipsoid $\mathcal{E}$, whose center of mass is also the origin.
Let $e > 0$ denote the $\e_1$-component of $\mathcal{E}  \cap \{ \text{positive } \e_1\text{-axis} \}$.
Abusing notation, we let $z_1 > 0$ and $- l z_1$ denote the $\e_1$-component of $z_1$ and $z_2$.
In turn, $e \leq z_1, l z_1 \leq n^{3/2}e$, from which it follows that $n^{-3/2} \leq l \leq n^{3/2}$, as desired.
\end{proof}

\begin{lemma}
\label{lem: max of u and affine in centered sec}
Let $S_h^c(z)$ be a centered section for $u$ at $z \in \R^n$.
Then
\[
h \leq \max_{S_h^c(z)} (\ell - u) \leq (1+n^{\frac{3}{2}})h.
\]
\end{lemma}

\begin{proof}
Without loss of generality, $z = 0$.

Since $(\ell - u)(0) = h$, by definition, the first inequality is trivial.

Now let $z_h \in S_h^c(0)$ be a point at which $\ell - u$ achieves its maximum value; let $z_1$ and $z_2$ be the opposite points on $\partial S_h^c(0)$ for which the segment $[z_1,z_2]$ contains $0$ and $x_h$; and let $\psi \geq 0$ be the one-dimensional concave function defined by $\ell - u$ on $[z_1,z_2]$.
Notice that the lines $\ell_i \subset \R^2$ determined by $(z_i,0)$ and $(0,h)$, for $i = 1,2$, are secant lines for graph of $\psi$.
Hence, the graph of $\psi$ must live under the union of the subgraphs of these two lines.
Assume that $z_h \in [0, z_2]$.
Otherwise, swap the roles of $z_1$ and $z_2$ in what follows.
Consider the triangle (in $\R^2$) over $[z_1,0]$ with height $h$ determined by the points $(z_1,0)$, $(0,0)$, and $(0,h)$.
Its maximal self-similar enlargement over $[z_1,z_2]$, whose base has right end point $(z_2,0)$ instead of $(0,0)$, has height $Ch$ with $1 + n^{-3/2} \leq C \leq 1 + n^{3/2}$ (see Lemma~\ref{lem: centered sections are balanced}), from which the second inequality follows.
\end{proof}

\section{Boundary Regularity of Maps with Convex Potentials Revisited}
\label{sec: Caff Bdry Reg Revisit}

In this section, we prove Theorem~\ref{thm: global Csigma}, and list some geometric properties of convex potentials defining optimal transports between absolutely continuous doubling measures on convex domains.
Let $u_0 : \R^n \to \R$ be a convex potential defining the optimal transport of Theorem~\ref{thm: global Csigma}.
It will be convenient to replace $u_0$ with its minimal convex extension outside of $X$.
More precisely, we consider the function
\[
u(x) := \sup_{z \in X, p \in \partial u_0(z)} \{ u_0(z) + p \cdot (x-z) \}.
\]
Similarly, we let $v_0$ be the Legendre transform of $u_0$, and
\[
\text{$v :=$ the minimal convex extension outside $Y$ of $v_0$}.
\]
Thus, $\nabla v$ is the optimal transport taking $g$ to $f$.

Given a centered section $S = S^c_h(z)$ for $u$, which exists at every $z \in \R^n$ (\cite{C2}), we define the normalized pair $(\tilde{u},\tilde{S})$ by
\[
\tilde{u}(x) := \frac{[u - \ell](A^{-1}x)}{h}
\text{ and }
\tilde{S} := A(S)
\]
where $A(\mathcal{E}) = B_1(0)$ and $\mathcal{E}$ is the John ellipsoid of $S$.
Moreover, we let $\tilde{f}$ and $\tilde{g}$ be the appropriate rescalings of $f$ and $g$ which ensure that $(\nabla \tilde{u})_{\#} \tilde{f} = \tilde{g}$.
Similarly, we define $\tilde{X} := A(X)$ and $\tilde{Y} := h^{-1}\nabla A^{-t}(Y)$.
Here and in the remainder of this work, we let 
\[
L^{-t} = (L^{-1})^t,
\] 
i.e., the transpose of the inverse of $L$, for any invertible transformation.

We first recall that optimal transports balance mass (\cite{B,V1}).

\begin{lemma}
\label{lem: Brenier soln}
Let $u : \R^n \to \R$ be convex and such that $(\nabla u)_\# f = g$, where $f$ and $g$ are two absolutely continuous measures. 
Then, for all Borel sets $B \subset \R^n$,
\[
\int_B f = \int_{\partial u(B)} g.
\]
\end{lemma}

Next we prove an Alexandrov maximum principle for normalized pairs.

\begin{lemma}
\label{lem: AMP}
There is an increasing function $\vartheta : [0,\infty) \to [0,\infty)$, depending only on dimension and the doubling constants of $f$ and $g$, such that $\lim_{d \to 0} \vartheta(d) = 0$ and
\[
|\tilde{u}(x)| \leq \vartheta(d_{\partial \tilde{S}}(x)) \text{ for } x \in \tilde{S}.
\]
Here $\tilde{S}$ is any normalized centered section based at any point in $\overline{X}$.
\end{lemma}

\begin{proof}
For notational simplicity, we suppress the dependence on $x$ and set
\[
d = d_{\partial \tilde{S}}(x) \text{ and } |\tilde{u}| = |\tilde{u}(x)|.
\]
First, observe that
\[
\partial \tilde{u}(\tfrac{1}{2}\tilde{S}) \subset B_R(0).
\]
Also, considering the cone generated by $(x, \tilde{u}(x))$ and $\tilde{S}$,
\[
\partial \tilde{u}(\tilde{S}) \supset K := \conv ( B_{r|\tilde{u}|}(0) \cup \tfrac{|\tilde{u}|}{d}\e)
\]
for some unit vector $\e = \e(x)$ and two positive constants $R$ and $r$ depending only on dimension.
Since the slope of the plane which determines $S^c_h$ is in $Y$, $0 \in \tilde{Y}$.
By assumption, the center of $\tilde{S}$ is in the closure of $\tilde{X}$.
In turn,
\[
\int_K \tilde{g} \leq \int_{\partial \tilde{u}(
\tilde{S})} \tilde{g} = \int_{\tilde{S}} \tilde{f} \leq C \int_{\frac{1}{2}\tilde{S}} \tilde{f} = C \int_{\partial \tilde{u}(\frac{1}{2}\tilde{S})} \tilde{g} \leq C \int_{B_R(0)} \tilde{g}.
\]
(Normalization affects neither the doubling property nor the doubling constants.)

Now let $B_{r_m}(t_m \e) \subset \conv( B_r(0) \cup \frac{1}{d}\e)$ for $m = 1, \dots, M$ be a sequence of balls chosen so that 
\[
\tfrac{1}{2}K_m \subset K_m \setminus K_{m-1} \text{ with } K_m := \conv ( B_{r_m}(t_m\e) \cup B_r(0)) \text{ and } K_0 := B_r(0). 
\]
By construction, $\{ \frac{1}{2}K_m \}_{m = 1}^M$ is a disjoint family, and
\[
M = M(d) \to \infty \text{ as } d \to 0.
\]
Hence, if we consider the collection
\[
\{ B_{r_m|\tilde{u}|}(t_m|\tilde{u}| \e) \}_{m=1}^M \subset K \subset \tilde{Y},
\]
and redefine $K_m := \conv ( B_{r_m|\tilde{u}|}(t_m|\tilde{u}|\e) \cup B_{r|\tilde{u}|}(0))$, we see that the redefined family $\{ \frac{1}{2}K_m \}_{m = 1}^M$ is also disjoint.
So
\[
M\int_{B_{r|\tilde{u}|}(0)} \tilde{g} \leq \sum_{m = 1}^{M} \int_{K_m} \tilde{g} \leq C \sum_{m = 1}^{M} \int_{\frac{1}{2} K_m} \tilde{g} \leq C \int_{K} \tilde{g}.
\]

Combing the two chains of inequalities above, we find that
\[
M \int_{B_{r|\tilde{u}|}(0)} \tilde{g} \leq C \int_{B_R(0)} \tilde{g} \leq C^k \int_{B_{r|\tilde{u}|}(0)} \tilde{g}
\]
where
\[
k = \bigg{\lceil} \frac{\log \frac{r}{R} |\tilde{u}|}{\log \frac{1}{2}}  \bigg{\rceil}.
\]
In turn, $M \leq C^k$.
Solving for $|\tilde{u}|$ concludes the proof.
\end{proof}

With Lemma~\ref{lem: AMP} in hand, Theorem~\ref{thm: global Csigma} follows from Caffarelli's arguments (\cite{C2,C3}).
Indeed, we first find that centered sections based at points in $\overline{X}$ have an engulfing property.
To prove this property, we replace Caffarelli's modulus for normalized solutions $Cd^{1/n}$, i.e, the classical Alexandrov maximum principle modulus, with the modulus $\vartheta(d)$ from Lemma~\ref{lem: AMP} in his proof.

\begin{lemma}
\label{lem: engulfing}
For any pair of constants $0 \leq \underline{t} < \overline{t} \leq 1$, there exists a constant $0 < t_0 \leq 1$ such that
\[
S^c_{t_0 h}(z) \subset \overline{t} S^c_h(x)
\]
for all $x \in \overline{X}$ and all $z \in \underline{t} S^c_h(x) \cap \overline{X}$.
The constant $t_0$ depends on $\underline{t}$, $\overline{t}$, dimension, and the doubling constants of $f$ and $g$.
\end{lemma}

Second, we obtain that $u$ is strictly convex in $\overline{X}$ (cf., \cite[Corollary~2.3]{C3}).
And, by duality and iteration, Theorem~\ref{thm: global Csigma}.
 
\begin{corollary}
\label{cor: strict convexity}
A constant $c > 0$, depending only on dimension and the doubling constants of $f$ and $g$, exists for which
\[
u(z) \geq u(x) + p \cdot (z - x) + c h
\]
for all $x \in \overline{X}$, $p \in \partial u(x)$, and all $z \in \partial S^c_h(x) \cap \overline{X}$.
\end{corollary}

\begin{proof}[Proof of Theorem~\ref{thm: global Csigma}]
Let $y,z \in \overline{Y}$, and let $h > 0$ be such that $z \in \partial S^c_h(v,y)$.
By Corollary~\ref{cor: strict convexity}, for any $q \in \partial v(y)$,
\[
v(z) \geq v(y) + q \cdot (z - y) + c h.
\]
By compactness, $ S_1^c(v,y) \subset B_{1/\tau}(y)$ for some $\tau \in (0,1)$ depending only on the inner and outer diameters of $X$ and $Y$.
(See \cite{C2}.)

Applying Lemma~\ref{lem: engulfing} iteratively, we see that
\[
S_{t_0^j}^c(v,y) \subset \tfrac{1}{2^j}S_1^c(v,y) \text{ for all } j \in \N \text{ for some } t_0 \in (0,1).
\]
Let $k \in \N$ such that $t_0^{k+1} \leq h \leq t_0^k$.
Then 
\[
z \in S_{t_0^k}^c(v,y) \text{ and } |y-z| \leq \frac{1}{\tau 2^k}.
\]
In turn, for $M \geq \log t_0/\log (1/2)$, we deduce that
\[
v(z) \geq v(y) + q \cdot (z - y) + ct_0 \tau |y-z|^M.
\]
Therefore, as $v|_{\overline{Y}}$ agrees with the Legendre transform of $u$ in $Y$, $u \in C^{1 + \sigma}(\overline{X})$ for some $\sigma \in (0,1)$, as desired.
\end{proof}

Following the proof of \cite[Corollary~2.2]{C3}, we find a first volume product estimate.

\begin{corollary}
\label{cor: sections are ellipsoids}
Let $x \in \overline{X}$ and $S^c_h(x)$ be a centered section for $u$ based at $x$.
There is a constant $r >0$, depending on dimension and the doubling constants of $f$ and $g$, such that
\[
B_r(0) \subset \nabla \tilde{u}(\tilde{S}) \subset B_{1/r}(0).
\]
Consequently,
\[
r^nh^n \leq |S^c_h(x)||\nabla u(S^c_h(x))| \leq \frac{h^n}{r^n}.
\]
\end{corollary}

\begin{remark}
An implication of Corollary~\ref{cor: sections are ellipsoids} is that
\[
\vartheta(d) = C d^{\frac{1}{n}},
\]
for some $C > 0$ depending only on dimension and the doubling constants of $f$ and $g$, is a candidate modulus for Lemma~\ref{lem: AMP}.
Indeed, since $r^n \leq |\nabla \tilde{u}(\tilde{S})| \leq r^{-n}$, this follows from Alexandrov's maximum principle.
\end{remark}

From Corollary~\ref{cor: strict convexity}, we also deduce that centered sections and classical sections
\[
S_h(u,x,p) := \{ z \in \overline{X} : u(z) < u(x) +  p \cdot (z-x) + h \},
\]
where $p \in \partial u(x)$, are comparable (see \cite{CLW1} for a proof).
When $u$ is differentiable at $x$, the set $\partial u(x)$ is a singleton, and we write $S_h(u,x)$ rather than $S_h(u,x,\nabla u(x))$.
We often also suppress the dependence on $u$.

\begin{corollary}
\label{cor: equivalence of sections}
There are constants $c > 0$, depending only on dimension, and $C \geq 1$, depending only on dimension and the doubling constants of $f$ and $g$, such that
\[
S^c_{ch}(x) \cap \overline{X} \subset S_h(x) \subset S^c_{Ch}(x) \cap \overline{X}
\]
for all $x \in \overline{X}$.
\end{corollary}

In addition, we find that the image of (centered) sections of $u$ of height $h$ are comparable to (centered) sections of $v$ height $h$.

\begin{corollary}
\label{cor: image of section is comp to dual section}
There are constants $c > 0$ and $C \geq 1$, depending only on dimension and the doubling constants of $f$ and $g$, such that
\[
S_{ch}(v,\nabla u(x)) \subset \nabla u(S_h(u,x)) \subset S_{Ch}(v,\nabla u(x))
\]
and
\[
S_{ch}^c(v,\nabla u(x)) \cap \overline{Y} \subset \nabla u(S_h^c(u,x)) \subset S_{Ch}^c(v,\nabla u(x)) \cap \overline{Y}
\]
for any $x \in \overline{X}$.
\end{corollary}

\begin{proof}
Up to a translation, we assume that $x = 0$.
Furthermore, up to subtracting off the tangent plane to $u$ at $0$, we assume that $u(0) = 0$ and $u \geq 0$. 

We start with the second inclusion.
Since $v$ in $\overline{Y}$ agrees with the Legendre transform of $u$ and $\nabla u(\overline{X}) = \overline{Y}$, $\nabla v(\nabla u(x)) = x$ for all $x \in \overline{X}$.
In particular, $\nabla v(\nabla u(0)) = 0$.
Moreover, $v(0) = 0$ and $v \geq 0$.
Considering Corollary~\ref{cor: equivalence of sections} then, it suffices to show that $v(\nabla u(x)) < Ch$ for all $x \in S^c_h(u,0) \cap \overline{X}$, which follows from Corollary~\ref{cor: sections are ellipsoids}.
Indeed, letting $A$ be the John transformation that normalizes $S^c_h(u,0)$, observe that
\[
v(\nabla u(x)) =  \nabla u(x) \cdot x  - u(x) = h \nabla \tilde{u}(Ax) \cdot Ax  - u(x) + \ell(x) < Ch,
\]
as desired.
(Recall Lemma~\ref{lem: max of u and affine in centered sec}.)

The first inclusion now follows from symmetry and duality.
Specifically, reversing the roles of $u$ and $v$ in the second inclusion and applying $\nabla u$, we see that
\[
S_h(v,0) = \nabla u( \nabla v(S_h(v,0))) \subset \nabla u( S_{Ch}(u,0) ).
\]
Replacing $h$ by $C^{-1}h$ concludes the proof.
\end{proof}

Finally, again following the arguments of \cite[Section~3]{C3}, we obtain a uniform density estimate on centered sections as well as a second volume product estimate, this time on (centered) sections and their images when $X$ is polynomially convex.
For completeness, let us recall the definition of polynomially convex and an important remark, both taken directly from \cite{C3}, which will be used in the proof of Theorem~\ref{thm: global C2 in plane}.

\begin{definition}
A domain $X$ is polynomially convex at the origin provided $0 \in \partial X = \{ x_n = \Gamma_X(x') \}$ (up to a rotation) and two constants $0 < \kappa_1, \kappa_2 < 1$ exist such that
\[
\Gamma_X(x') \leq \frac{x' \cdot z'}{|z'|^2} \Gamma_X(z')
\]
whenever $|x'|,|z'| \leq \delta$ and $x'$ lies in the truncated cone
\[
\begin{cases}
|x'| < \kappa_1|z'| \\
|\sin \angle(x',z')| \leq \kappa_2.
\end{cases} 
\]
Here $\angle(x',z')$ denotes the angle between $x'$ and $z'$.
A domain is polynomially convex if it is polynomially convex at every point on its boundary.
\end{definition}

\begin{remark}
In two dimensions, every convex domain is polynomially convex.
In every dimension, given a polynomially convex domain, the constants $\kappa_1$ and $\kappa_2$ can be chosen uniformly for all points on its boundary depending only on the inner and outer diameters of the domain. (See \cite[Section 3, Remark 2 and Lemma 3.1]{C3}.)
\end{remark}

\begin{proposition}
\label{prop: uni density}
Let $X$ be polynomially convex. 
Then there are constants $C \geq 1$ and $c > 0$, depending on the inner and outer diameters of $X$ and $Y$, dimension, and the doubling constants of $f$ and $g$, such that
\[
C\frac{|S^c_h(x) \cap X|^{\frac{1}{n}}}{|S^c_h(x)|^{\frac{1}{n}}} \geq \frac{\diam(S^c_h(x) \cap X)}{\diam(S^c_h(x))} \geq c
\]
for any $x \in \partial X$.
\end{proposition}

For the convenience of the reader, we sketch the proof of this proposition.
But first we make a remark.

\begin{remark}
The polynomial convexity of $X$ only plays a role in proving the first inequality, between normalized volume and normalized diameter, but for ellipsoids centered at points in $\partial X$ rather than centered sections.
This inequality is one about convex sets, and nothing more.
In terms of the geometry of $X$ and $Y$, the remainder of the proof uses only that $X$ and $Y$ are convex, and have boundaries that can be locally written as graphs of a Lipschitz functions.
\end{remark}

We write $A \sim B$ if $cB \leq A \leq CB$ for some $c > 0$ and $C > 0$.

\begin{proof}[Sketch of Proof]
Let $0 \in \partial X$ and consider $S = S_h^c(0)$.
The first inequality is a consequence of \cite[Lemma~3.2]{C3} and the comparability of $S$ to an ellipsoid $\mathcal{E}$ centered at $0$; and so, it suffices to show that the normalized diameter of $S$, 
\[
\delta := \frac{\diam (S \cap X)}{\diam (S)},
\]
cannot be too small.

From Corollary~\ref{cor: sections are ellipsoids}, if $r_i$  is the principle radius of $\mathcal{E}$ in the $\e_i$ direction, then $\nabla u(S)$ is comparable to an ellipsoid $\mathcal{E}^\ast$ that has principle radius $h/r_i$ in the direction $\e_i$.
Let $y = \nabla \ell \in Y$, where $\ell$ defines $S$.
Up to a rotation, we assume that $r_1 \geq r_i$ for all $i \neq 1$.
Let $x_\pm \in \partial S$ be such that $\nu_{\partial S}(x_+) = \e_1$ and $\nu_{\partial S}(x_-) = -\e_1$.
In particular, $\nabla (\ell - u)$ at $x_\pm$ is parallel to $\e_1$.
Since $|x_\pm| \sim r_1$, $x_\pm \notin X$ if $\delta > 0$ is sufficiently small.
Hence, $y_\pm = \nabla u (x_\pm) = y \pm t_\pm \e_1 \in \partial Y$ with $t_\pm \sim h/r_1$.

Now $\partial Y$ is locally the graph of a Lipschitz function in the direction $\nu$.
Let $t_1 > 0$ be the largest constant for which $y_1 := y - t_1\nu \in \nabla u(S)$.
Thus, $y_1 \in \partial Y$, as $\nabla u(S) \subset \overline{Y}$.
Since $y_\pm \in \partial Y$ and $\partial Y$ is Lipschitz function in the direction $\nu$, we find the inequality $t_1 \leq Ch/r_1$.
Let $t_2 > 0$ be such that $y_2 := y + t_2\nu \in Y \cap \partial (\nabla u(S))$.
By Corollary~\ref{cor: sections are ellipsoids}, it follows that $t_2 \leq Ch/r_1$.
Thus, if $x_2 = (\nabla u)^{-1}(y_2)$, then $x_2 \in \partial S \cap X$. 
Moreover, by convexity and Lemma~\ref{lem: max of u and affine in centered sec},
\[
|x_2||y - y_2|  = |x_2||\nabla (\ell - u)(x_2)| \geq (\ell - u)(0) = h.
\]
In turn, $|x_2| \geq r_1/C$.
But this contradicts $\delta > 0$ being arbitrarily small.
\end{proof}

\begin{corollary}
\label{cor: volume product}
Let $X$ be polynomially convex.
Then constants $C \geq 1$ and $c> 0$ exist, depending on the inner and outer diameters of $X$ and $Y$, dimension, and the doubling constants of $f$ and $g$, for which
\[
c h^n \leq |S^c_h(x) \cap X||\nabla u(S^c_h(x))| \sim |S_h(x)||\nabla u(S_h(x))| \leq Ch^n
\]
for any $x \in \partial X$.
\end{corollary}

%~~~~~~~~~~~~~~~~~~~~~~~~~~~~~~~~~~~~~~~~~~~~~~~~~~~~~~~~~~~~%
\section{The Flat Case}
\label{sec: The Flat Case}

Let $u$ be convex and such that 
\begin{equation}
\label{eqn: global OT flat}
(\nabla u)_{\#} d_{\partial \{x_n > 0\}}^\alpha =  d_{\partial \{y_n > 0\}}^\beta,
\end{equation}
for two nonnegative constants $\alpha$ and $\beta$.
Then, by the arguments of Section~\ref{sec: Caff Bdry Reg Revisit} and classical regularity theory for the Monge--Amp\`{e}re equation, we find that $u$ is strictly convex in $\{x_n \geq 0\}$, $u \in C^{1,\sigma}_{\rm loc}(\{x_n \geq 0 \}) \cap C^\infty_{\rm loc}(\{x_n > 0 \})$, $u_n \geq 0$, and solves 
\begin{equation}
\label{eqn: global MA flat}
\det D^2 u = \frac{x_n^\alpha}{u_n^\beta} \text{ in } \{ x_n > 0 \} \text{ and } u_n = 0 \text{ on } \{x_n = 0\}.
\end{equation}

From this point forward, in this section, we assume that $\max\{ \alpha, \beta\} > 0$.
If $\alpha = \beta = 0$, then $u$ must be a quadratic polynomial by the classical Liouville theorem for the Monge--Amp\`{e}re equation; indeed, its even reflection over the set $\{ x_n = 0\}$ solves $\det D^2 u = 1$ in $\R^n$.

First, we prove a Pogorelov estimate in $x'$, which holds up to $\{ x_n = 0 \}$. 

\begin{proposition}
\label{prop: pogo tang}
Let $u$ be convex and satisfy $u_n \geq 0$ in $\{x_n > 0\}$ and \eqref{eqn: global MA flat}.
Let $x_0 \in \{x_n = 0\}$ and $\ell_{x_0}$ be the tangent plane to $u$ at $x_0$.
For any tangential direction $\e$, i.e., such that $\e \cdot \e_n = 0$,
\begin{equation}
\label{eqn: tang C11}
u_{\e\e}|u - \ell_{x_0}-h| \leq C(n+\beta, \|\partial_\e u - \partial_\e\ell_{x_0} \|_{L^\infty(S_h(x_0))}).
\end{equation}
\end{proposition}

\begin{proof}
Up to a translation and subtracting off the tangent plane to $u$ at $x_0$, we assume that $x_0 = 0$ and $u(0) = \nabla u(0) = 0$.
Now let $\epsilon > 0$, $\Upsilon := B \cap \{ y_n > 0 \}$, for some large ball $B$ centered at the origin, and $\Omega$ be a dilation of $(\nabla u)^{-1}(\Upsilon)$ such that
\[
\int_\Omega (x_n + \epsilon)^\alpha = \int_\Upsilon (y_n + \epsilon)^\beta.
\]
Furthermore, let $\nabla \psi$ be the optimal transport taking $f = (x_n + \epsilon)^\alpha \mres \Omega$ to $g = (y_n + \epsilon)^\beta \mres \Upsilon$.
Note that the even reflection in $x_n$ of $\psi$, call it $\bar{\psi}$, is a potential whose gradient is the optimal transport taking the even reflection of $f$ to the even reflection of $g$, in $x_n$ and $y_n$ respectively.
So points along $\{x_n = 0\}$ can be turned into interior points.
By \cite{C1} and symmetry, $D^2\bar{\psi}$ is locally H\"{o}lder continuous in $\bar{\Omega}$, the reflection of $\Omega$ over the $x_n$-axis, $\nabla \bar{\psi}(\Omega) \subset \{y_n > 0 \}$, and $\bar{\psi}_n = 0$ on $\bar{\Omega} \cap \{x_n = 0 \}$.
In particular, we find that
\begin{equation}
\label{eqn: MA for pogo tang approx}
\det D^2 \psi = \frac{(x_n + \epsilon)^\alpha}{(\psi_n + \epsilon)^\beta} \text{ in } \Omega,
\end{equation}
\begin{equation}
\label{eqn: bdry for pogo tang approx}
\psi_n = 0 \text{ on } \{ x_n = 0 \},
\end{equation}
and $\psi$ is smooth in $\Omega$ and $D^2 \psi$ is H\"older continuous and strictly positive definite up to $\{ x_n = 0 \}$.
By Theorem~\ref{thm: global Csigma}, $\nabla \psi$ converges to $\nabla u$ locally uniformly in $\bar{\Omega} \cap \{ x_n \geq 0\}$.
(We can choose doubling constants for the denisites $(x_n + \epsilon)^\alpha$ and $(y_n + \epsilon)^\beta$ unifornly in $\epsilon$.)
So it suffices to prove \eqref{eqn: tang C11} for $\psi$.

If we differentiate the log of \eqref{eqn: MA for pogo tang approx} and \eqref{eqn: bdry for pogo tang approx} in any tangential direction $\e$ (with $\e \cdot \e_n = 0$), we have that
\begin{equation}
\label{eqn: MA_1 for pogo approx}
\psi^{ij}\partial_{ij}\psi_\e = -\beta \frac{\psi_{n\e}}{\psi_n + \epsilon} \text{ in } \{ x_n > 0 \}
\end{equation}
and
\[
\label{eqn: MA_1 for pogo approx bdry}
\partial_n \psi_\e = 0 \text{ on } \{ x_n = 0 \}. 
\]
The right-hand side of this equation is H\"older continuous.
Hence, $D^2\psi_\e$ is H\"older continuous across $\{x_n = 0 \}$.
Differentiating again in the $\e$ direction, we find that
\begin{equation}
\label{eqn: MA_11 for pogo approx}
\psi^{ij}\partial_{ij} \psi_{\e\e} = \beta \frac{\psi_{n \e}^2}{(\psi_n + \epsilon)^2} - \beta \frac{\psi_{n\e\e}}{\psi_n + \epsilon} + \psi^{ik}\psi^{jl}\psi_{ij\e}\psi_{kl\e}
\end{equation}
and
\[
\partial_n \psi_{\e\e} = 0 \text{ on } \{ x_n = 0 \}. 
\]
The right-hand side of this equation is H\"older continuous, given the H\"older continuity of $D^2 \psi_\e$ just observed.
In conclusion, the fourth order derivatives $\bar{\psi}_{ijkl}$ are continuous across $\{ x_n = 0 \}$ provided no more than two of the four indices are $n$.
And so,
\[
M := \log |\bar{\psi}| + \log \bar{\psi}_{\e\e} + \tfrac{1}{2}\bar{\psi}_\e^2
\]
is $C^2(\overline{S})$ with $S = S_h(\bar{\psi},0)$.
The ball $B$ is chosen large enough so that $S_h(\bar{\psi},0) \Subset \bar{\Omega}$.
Furthermore, up to subtracting off the tangent plane to $\bar{\psi}$ and $h$, we can assume that that $S = \{ \bar{\psi} < 0 \}$.

Let $z \in \{ \bar{\psi} < 0 \}$ be a point at which $M$ achieves its maximum.
(The point $z \notin \{ \bar{\psi} = 0 \}$ since $e^M$ vanishes on $\{ \bar{\psi} = 0 \}$.)
As $\bar{\psi}$ is even in $x_n$, we can assume that $z \in \{ x_n \geq 0 \}$.
So for notational simplicity, we identify $\bar{\psi}$ with $\psi$.

\medskip
{\bf Case 1: $M$ is achieved at $z \in \{ x_n = 0 \}$.}
By \eqref{eqn: bdry for pogo tang approx}, at $z$,
\[
\psi_{ni} = 0 \text{ for all } i < n.
\]
So, after an orthogonal transformation in the tangential coordinates, which leaves the equation invariant, we can assume that $D^2\psi(z)$ is diagonal and $\e = \e_1$.
(While the equality $\psi_{ni}|_{\{ x_n = 0 \}} = 0$ simplifies some of the expressions below, we refrain from using it, so that Case 2 becomes evident.)

First, differentiating $M$ twice in the $\e_i$ and evaluating at $z$, we find
\begin{equation}
\label{eqn: pogo tang max equal approx}
\frac{\psi_i}{\psi} + \frac{\psi_{11i}}{\psi_{11}} + \psi_1\psi_{1i} = 0
\end{equation}
and
\[
\frac{\psi_{ii}}{\psi} - \frac{\psi_i^2}{\psi^2}  + \frac{\psi_{11ii}}{\psi_{11}} - \frac{\psi_{11i}^2}{\psi_{11}^2} + \psi_{1i}^2 + \psi_1\psi_{1ii} \leq 0.
\]
So multiplying by $\psi^{ii} = \psi_{ii}^{-1}$ and summing over $i$, we deduce that
\begin{equation}
\label{eqn: pogo tang max ineq approx}
\frac{n}{\psi} - \frac{\psi^{ii}\psi_i^2}{\psi^2}  + \frac{\psi^{ii}\psi_{11ii}}{\psi_{11}} - \frac{\psi^{ii} \psi_{11i}^2}{\psi_{11}^2} + \psi_{11} + \psi_1 \psi^{ii} \psi_{1ii} \leq 0.
\end{equation}
Second, considering \eqref{eqn: MA_1 for pogo approx}, \eqref{eqn: MA_11 for pogo approx}, and \eqref{eqn: pogo tang max ineq approx}, we have that
\[
\frac{n}{\psi} - \frac{\psi^{ii}\psi_i^2}{\psi^2} -  \frac{\beta \psi_{n11}}{(\psi_n + \epsilon) \psi_{11}} + \frac{ \psi^{ii}\psi^{jj}\psi_{1ij}^2}{\psi_{11}} - \frac{\psi^{ii} \psi_{11i}^2}{\psi_{11}^2} + \psi_{11} -\frac{\beta \psi_1 \psi_{n1} }{\psi_n + \epsilon} \leq 0.
\]
Third,
\[
\frac{ \psi^{ii}\psi^{jj}\psi_{1ij}^2}{\psi_{11}} - \frac{\psi^{ii} \psi_{11i}^2}{\psi_{11}^2} = \frac{1}{\psi_{11}} \sum_{i = 2, j =1 }^n \frac{\psi_{1ij}^2}{\psi_{ii}\psi_{jj}},
\]
and, by \eqref{eqn: pogo tang max equal approx}, we find that
\[
\frac{ \psi^{ii}\psi^{jj}\psi_{1ij}^2}{\psi_{11}} - \frac{\psi^{ii} \psi_{11i}^2}{\psi_{11}^2} \geq \sum_{i = 2}^n \frac{\psi_i^2}{ \psi_{ii} \psi^2}.
\]
Our first three steps together yield
\[
\frac{n}{\psi} - \frac{\psi_1^2}{\psi_{11}\psi^2} - \frac{\beta}{\psi_n+\epsilon}\bigg(\frac{\psi_{n11}}{\psi_{11}} + \psi_1\psi_{n 1} \bigg)  + \psi_{11} \leq 0.
\]
Now from \eqref{eqn: pogo tang max equal approx} again, i.e.,
\[
\frac{\psi_{11n}}{\psi_{11}} + \psi_1\psi_{n 1} = -\frac{\psi_n}{\psi},
\]
it follows (recall $\psi < 0$ and $\psi_n = 0$ at $z$) that
\[
\frac{n}{\psi} - \frac{\psi_1^2}{\psi_{11}\psi^2} + \psi_{11}  = \frac{1}{\psi}\bigg( n + \beta \frac{\psi_n}{\psi_n + \epsilon}\bigg) - \frac{\psi_1^2}{\psi_{11}\psi^2} + \psi_{11} \leq 0.
\] 
Consequently,
\[
|\psi|\psi_{11} \leq C(n, \|\psi_1\|_{L^\infty(S)}).
\]

\medskip
{\bf Case 2: $M$ is achieved at $z \in \{ x_n > 0 \}$.}
In this case, after a rotation $\e \mapsto \e_1$, a shearing transformation $x \mapsto (x_1 - s_i x_i,x_2,\dots, x_n)$ for $i = 2, \dots, n$, and then a rotation in the $x_i$ variables for $i > 1$, we can assume $D^2 \psi(z)$ is diagonal provided we replace \eqref{eqn: MA for pogo tang approx} and \eqref{eqn: bdry for pogo tang approx}
with
\begin{equation}
\label{eqn: MA for pogo tang approx int}
\det D^2 \psi = \frac{(x \cdot {}^\backprime\xi + \epsilon)^\alpha}{(\psi_{\xi} + \epsilon)^\beta} \text{ in } \{ x \cdot {}^\backprime\xi > 0 \} \text{ with } {}^\backprime\xi \cdot \e_1 = 0,\, |\xi| = 1,
\end{equation}
and
\begin{equation}
\label{eqn: bdry for pogo tang approx int}
\psi_{\xi} := s_n\psi_1 +  \psi_{{}^\backprime\xi} = s_n\psi_1 +  \nabla \psi \cdot {}^\backprime\xi = 0 \text{ on } \{ x \cdot {}^\backprime\xi = 0 \}.
\end{equation}
Identical computations to those in Case 1 yield the same final inequality:
\[
\frac{n + \beta}{\psi} - \frac{\psi_1^2}{\psi_{11}\psi^2} + \psi_{11}  \leq \frac{1}{\psi}\bigg( n + \beta \frac{\psi_\xi}{\psi_\xi + \epsilon}\bigg) - \frac{\psi_1^2}{\psi_{11}\psi^2} + \psi_{11} \leq 0,
\] 
from which \eqref{eqn: tang C11} follows for $\psi$, as desired.
\end{proof}

A an important consequence of Proposition~\ref{prop: pogo tang} (applied to $u$ and $v$ the Legendre transform of $u$) is that $u_n$ and $x_n^\gamma$ are comparable.
Recall,
\[
\gamma := \frac{1+\alpha}{1+\beta}.
\]
Since the rescalings
\begin{equation}
\label{eqn: rescalings}
u_t(x) := \frac{u(D_t x)}{t} \text{ with } D_t := \diag(t^{\frac{1}{2}}\Id', t^{\frac{1}{1+\gamma}})
\end{equation}
leave the equation invariant, the correct geometry in which to work is defined by the cylinders
\[
\mathcal{C}_r(z) := B'_{r^{1/2}}(z') \times (z_n - r^\frac{1}{1+\gamma}, z_n + r^\frac{1}{1+\gamma}) \text{ and } \mathcal{C}_r := \mathcal{C}_r(0).
\]
We state our comparability estimate in this geometry.

\begin{lemma}
\label{lem: pure normal comp to 1}
Let $u$ be convex and satisfy \eqref{eqn: global OT flat}.
Then two constants $c_0 > 0$ and $C_0 > 0$, depending on $\|\nabla_{x'} u\|_{L^\infty(\mathcal{C}_1 \cap \{ x_n \geq 0\})}$, $\|\nabla_{y'} v\|_{L^\infty(\nabla u(\mathcal{C}_1 \cap \{ x_n \geq 0\}))}$, $\alpha$, $\beta$, and $n$, exist such that
\begin{equation}
\label{eqn: normal C11}
c_0 \leq  \frac{u_n}{x_n^\gamma} \leq C_0  \text{ on } \mathcal{C}_1 \cap \{ x_n \geq 0\}.
\end{equation}
\end{lemma}

\begin{proof} 
We claim that if
\[
\text{$D^2_{x'} u(0) \geq \frac{1}{M}\Id'$ (or $\leq M \Id'$)},
\]
then
\[
\text{$u_n(0,x_n) \leq C(M)x_n^\gamma$ (or $\geq c(M)x_n^\gamma$).}
\]
Before proving this claim, we use it to conclude the proof of our lemma.
Up to subtracting the tangent plane to $u$ at $0$, assume that $u \geq 0$.
Then, as $0$ was arbitrary, our lemma follows, since Proposition~\ref{prop: pogo tang} (applied to $u$ and $v$ the Legendre transform of $u$) tells us that 
\[
\frac{1}{M}\Id' \leq D^2_{x'}u \leq M\Id' \text{ in } B'_1 \times \{ 0 \}.
\]

Now we prove our claim. 
First, we show that $S_h(0)$ is comparable to an ellipsoid whose axes are parallel to the coordinate axes.
Indeed, let 
\[
l_{i;h} = l_{i;h}(u,0) := - \inf \{ x_i : (x_i,0) \in S_h(u,0) \}
\]
and
\[
r_{i;h} = r_{i;h}(u,0) := \sup \{ x_i : (x_i,0) \in S_h(u,0) \}.
\]
By Corollary~\ref{cor: equivalence of sections},
\[
c l_{i;h} \leq r_{i;h} \leq \frac{l_{i;h}}{c}
\]
for some $c \leq 1$, depending only on $\alpha$, $\beta$, and $n$.
Now define
\[
w_{i;h} := r_{i;h} + l_{i;h}
\]
and
\[
d_h = d_h(u,0) := \sup \{ x_n : (0,x_n) \in S_h(u,0) \},
\]
and consider 
\[
T_{i;h} := \text{\,the triangle determined by $(0,d_h)$ and $(\pm \min \{r_{i;h},l_{i;h}\} ,0)$.}
\]
(The center of mass of $T_{i;h}$ is $(0, \frac{1}{3}d_h)$.)
Note that $S_h(0) \cap \Span\{ \e_i,\e_n \}$ is contained in the union of the subgraphs of the lines determined by $(0,d_h)$ and  $(\pm \min \{r_{i;h},l_{i;h}\},0)$ and inside the strip $[-l_{i;h}, r_{i;h}] \times [0,\infty)$.
The heights of the intersections of these lines and the boundary of the strip is less than or equal to $d_h(1 + c^{-1})$, from which find a $C \geq 1$, depending only on $\alpha$, $\beta$, and $n$, such that $C T_{i;h} \supset S_h(0) \cap \Span\{ \e_i,\e_n \}$.
In turn, if $\mathcal{E}_h$ is the John ellipsoid of $\conv \{ \cup_{i < n} T_{i;h} \}$, then
\begin{equation}
\label{eqn: new ellipsoid}
\delta \mathcal{E}_h \subset S_h(0) \subset \frac{1}{\delta} \mathcal{E}_h,
\end{equation}
for some $0 < \delta < 1$ depending only on $c$.
The ellipsoid $\mathcal{E}_h$ has axes parallel to the coordinate axes, and is our desired ellipsoid.

Since $\mathcal{E}_h$ has axes parallel to the coordinate axes the distance from its center to $\partial \{x_n > 0\}$ is the vertical height of the center, which is comparable to $d_h$.
Then arguing as in Lemma~\ref{lem: doubling on ellipsoids}, but using the ellipsoid just constructed above, we see that
\[
\frac{1}{C} d_h^\alpha d_h w_{1;h} \cdots w_{n-1;h} \leq  \int_{S_h(u)} (x_n)_+^\alpha \leq Cd_h^\alpha d_h w_{1;h} \cdots w_{n-1;h} .
\]
And, as $S_h(v)$ is dual to $S_h(u)$ (Corollary~\ref{cor: image of section is comp to dual section}), we similarly find that
\[
\frac{1}{C} \frac{h^\beta}{d_h^\beta} \frac{h}{d_h} \frac{h}{w_{1;h}} \cdots  \frac{h}{w_{n-1;h} } \leq \int_{S_h(v)} (y_n)_+^\beta \leq C \frac{h^\beta}{d_h^\beta} \frac{h}{d_h} \frac{h}{w_{1;h}} \cdots  \frac{h}{w_{n-1;h} }.
\]
Moreover, by Lemma~\ref{lem: Brenier soln} and Corollary~\ref{cor: image of section is comp to dual section}, again,
\[
\int_{S_h(u)} (x_n)_+^\alpha \sim \int_{S_h(v)} (y_n)_+^\beta.
\]
Therefore,
\begin{equation}
\label{eqn: mass balance doubling}
\frac{1}{C} \leq \frac{d_h^{2 + \alpha +\beta} w_{1;h}^2 \cdots w_{n-1;h}^2}{h^{n+\beta}} \leq C.
\end{equation}

By assumption, for all $i \leq n-1$,
\[
w_{i;h}^2 \leq 2M h.
\]
In turn, $h^{1+\beta} \leq CM^{n-1}d_h^{2 + \alpha +\beta}$, or, equivalently,
\[
h^{\frac{1}{1+\gamma}} \leq CM^{\frac{n-1}{2 + \alpha +\beta}}d_h.
\]
Thus, $C(M)x_n^{1+\gamma} \geq u(0,x_n) \geq 0$.
So our claim follows by the convexity of $u$; indeed,
\[
C(M)2^{1+\gamma}x_n^{1+\gamma} \geq u(0,2x_n) \geq u(0,2x_n) - u(0,x_n) \geq u_n(0,x_n)x_n.
\]
\end{proof}

Lemma~\ref{lem: pure normal comp to 1} effectively gives us control over the second derivatives of $u$ in the normal direction.
And since $u_n/x_n^\gamma$ is a solution to an elliptic equation, with Lemma~\ref{lem: pure normal comp to 1} in hand, we can prove an oscillation decay estimate for $u_n/x_n^\gamma$.
In particular, 
\[
\phi := \frac{u_n}{x_n^\gamma}
\]
solves
\[
u^{ij}\phi_{ij} + \frac{\beta x_n\delta^{in} + (1+\gamma)u_nu^{in}}{x_nu_n}\phi_i = 0. 
\]
Here $\delta^{ij} = 0$ if $i \neq j$ and $\delta^{ij} = 1$ if $i = j$.

\begin{proposition}
\label{prop: osc decay at bdry}
Let $u$ be convex and satisfy \eqref{eqn: global OT flat}.
A constant $\zeta \in (0,1)$ exists, depending only on $\alpha$, $\beta$, and $n$, such that
\begin{equation}
\label{eqn: osc decay at bdry}
\osc_{\mathcal{C}_{1/2} \cap \{ x_n \geq 0\}} \frac{u_n}{x_n^\gamma}  \leq (1-\zeta) \osc_{\mathcal{C}_1 \cap \{ x_n \geq 0\}} \frac{u_n}{x_n^\gamma}.
\end{equation}
\end{proposition}

(A similar and simpler version of the proof of this proposition can be found in the proof of Lemma~\ref{lem: bdry does not flatten}.)

\begin{proof}
From \eqref{eqn: normal C11}, 
\[
\text{ either }\frac{u_n}{x_n^\gamma}(2^{-\frac{1}{1+\gamma}} \e_n) < \frac{C_0+c_0}{2}
\text{ or }
\frac{u_n}{x_n^\gamma}(2^{-\frac{1}{1+\gamma}} \e_n) \geq \frac{C_0+c_0}{2}.
\]
If the first inequality holds, we build a barrier that pulls $u_n/x_n^\gamma$ down in $\mathcal{C}_{1/2}$.
Whereas if the second inequality holds, we build a barrier that pulls $u_n/x_n^\gamma$ up in $\mathcal{C}_{1/2}$.

\medskip
{\bf Case 1: $\alpha < \beta$.}
If the first inequality holds, up to dividing by $C_0$, assume that $C_0 = 1$.
Then let
\[
\psi := (1-\epsilon)x_n^\gamma + c_1 \epsilon |x'|^2 - \epsilon x_n^\kappa
\]
with $c_1 > 0$ and $1 > \kappa > \gamma$ to be chosen.
(This is our upper barrier.)
First, note that
\[
\psi \geq x_n^\gamma \text{ on } \{ x_n^\gamma \leq \tfrac{c_1}{2}|x'|^2 \}.
\]
Second, given any $\delta \in (0,2^{-\frac{1}{1+\gamma}})$, applying the Harnack inequality to $u_n/x_n^\gamma$ along a chain of overlapping balls, we find that
\[
\sup_{\mathcal{C}_{1/2} \cap \{ x_n \geq \delta \}  } u_n \leq  (1 - c(\delta)) x_n^\gamma,
\]
for some small $c(\delta) > 0$.
Third, let $\delta$ be small enough to ensure that
\[
\Omega := \{ x_n^\gamma >\tfrac{c_1}{2}|x'|^2 \} \cap \{ x_n < \delta \} \subset \mathcal{C}_{1/2} \cap \{ x_n \geq 0\}.
\] 
Set
\[
\epsilon := c(\delta).
\]
Then 
\[
\psi \geq u_n \text{ on } \partial \Omega.
\]
With our boundary values understood, we now turn to the interior of $\Omega$ and the equation.

Suppose that $u_n - \psi$ achieves its maximum at some point in $\Omega$.
Then, at this point, $u_n = \psi + s$ for some $s \geq 0$, and
\begin{equation}
\label{eqn: osc decay rhs contra}
0 \geq u^{ij}\partial_{ij}(u_n - \psi) \geq \frac{\alpha}{x_n} - \frac{\beta \psi_n}{ \psi} - u^{ii}\psi_{ii}.
\end{equation}
Using that $\nabla u_n = \nabla \psi$ at our distinguished point, we show that \eqref{eqn: osc decay rhs contra} is impossible provided $\delta > 0$ is sufficiently small, which forces our distinguished point to be very close to $\{ x_n = 0 \}$.

First, we compute an upper bound for the quotient $\psi_n/\psi$.
Observe that
\[
\frac{\psi_n}{ \psi} \leq \frac{(1-\epsilon)\gamma x_n^{\gamma-1} - \epsilon\kappa x_n^{\kappa-1} }{(1-\epsilon)x_n^\gamma  - \epsilon x_n^\kappa} = \frac{\gamma}{x_n} \frac{(1-\epsilon)x_n^{\gamma} - \epsilon\kappa \gamma^{-1}x_n^{\kappa} }{(1-\epsilon)x_n^\gamma  - \epsilon x_n^\kappa}.
\]
Now we look at the sum $u^{ii}\psi_{ii}$, which we break into two pieces: $u^{nn}\psi_{nn}$ and the remainder.
First, observe that
\[
\sum_{i < n} u^{ii}\psi_{ii} = 2 c_1 \epsilon \trace (D^2_{x'} u)^{-1}.
\]
From Proposition~\ref{prop: pogo tang} (applied to $u$ and its Legendre transform), we have that
\[
\frac{1}{C'} \leq (D^2_{x'} u)^{-1} \leq C'.
\]
Therefore, choose
\[
c_1 := \frac{1}{2C'(n-1)}.
\]
Hence,
\[
- \sum_{i < n} u^{ii}\psi_{ii} \geq - \epsilon.
\]
Now considering $D^2u$ as a block matrix, we see that
\[
u^{nn} = \frac{1}{u_{nn} - \nabla_{x'} u_n (D^2_{x'} u )^{-1} \nabla_{x'} u_n} \geq \frac{1}{u_{nn}};
\]
and so, provided $\delta > 0$ is small enough so that $\psi_{nn} \leq 0$,
\[
- u^{nn}\psi_{nn} \geq - \frac{\gamma-1}{x_n} \frac{(1-\epsilon) x_n^\gamma -  \epsilon \kappa \gamma^{-1} (\kappa-1)(\gamma -1)^{-1} x_n^\kappa }{(1-\epsilon) x_n^\gamma - \epsilon \kappa \gamma^{-1} x_n^\kappa}.
\]
In turn, we have the following inequality, for the right-hand side of our equation,
\[
\frac{\alpha}{x_n} - \frac{\beta \psi_n}{ \psi} - u^{ii}\psi_{ii}  \geq \frac{\alpha}{x_n} - \frac{\beta\gamma}{x_n} {\rm I} - \frac{\gamma-1}{x_n} {\rm II} - \epsilon
\]
with
\[
{\rm I} := \frac{(1-\epsilon)x_n^{\gamma} -  \epsilon\kappa \gamma^{-1}x_n^{\kappa} }{(1-\epsilon)x_n^\gamma  -  \epsilon x_n^\kappa}
\]
and
\[
{\rm II} :=   \frac{(1-\epsilon) x_n^\gamma - \epsilon \kappa \gamma^{-1} (\kappa-1)(\gamma -1)^{-1} x_n^\kappa }{(1-\epsilon) x_n^\gamma - \epsilon \kappa \gamma^{-1} x_n^\kappa }.
\]
As
\[
\alpha - \beta \gamma = \gamma - 1,
\]
we rearrange our lower bound as follows, splitting ${\rm II}$ into two pieces: 
\[
\frac{\alpha}{x_n} - \frac{\beta\gamma}{x_n} {\rm I} - \frac{\gamma-1}{x_n} {\rm II} = \frac{\alpha}{x_n}( 1 - {\rm II} ) +  \frac{\beta\gamma}{x_n} ({\rm II} -{\rm I} ).
\]

Now we estimate the two factors, above, in parentheses.
Observe that
\[
1 - {\rm II} \geq - \frac{ \epsilon }{1-\epsilon}C_{0,\kappa,\gamma} x_n^{\kappa - \gamma}( 1 + \bar{C}\epsilon x_n^{\kappa - \gamma} )
\]
with
\[
C_{0,\kappa,\gamma} := \frac{\kappa(\kappa-\gamma)}{\gamma(1 - \gamma )} > 0.
\]
(Here and below $\bar{C} > 0$ is a large constant that may change from line to line; it depends only on $\kappa$ and $\gamma$.)
Similarly,
\[
{\rm II} -{\rm I} \geq \frac{ \epsilon }{1-\epsilon} C_{1,\kappa,\gamma} x_n^{\kappa - \gamma}( 1 - \bar{C}  \epsilon x_n^{\kappa - \gamma} )
\]
with
\[
C_{1,\kappa,\gamma} := \frac{(\kappa - \gamma)(1 + \kappa - \gamma)}{\gamma(1 - \gamma)} > 0.
\]

Finally, we conclude.
From above,
\[
\frac{\alpha}{x_n} - \frac{\beta\gamma}{x_n} {\rm I} - \frac{\gamma-1}{x_n} {\rm II} \geq \frac{ \epsilon }{1-\epsilon}\frac{C_{0,\kappa,\gamma}}{\kappa} x_n^{\kappa - \gamma-1} ( (\beta \gamma + \kappa)(1-\gamma) - \bar{C} \epsilon).
\]
Then, considering \eqref{eqn: osc decay rhs contra} and choosing $\epsilon > 0$ sufficiently small depending only on $\kappa$ and $\gamma$, we find the inequality
\[
0 \geq \delta^{\kappa - \gamma-1} - \bar{C},
\]
which is impossible once $\delta > 0$ is sufficiently small, as desired.

Consequently,
\[
\psi \geq u_n \text{ in } \mathcal{C}_{1/2} \cap \{ x_n \geq 0\}.
\]
In particular,
\[
(1-\epsilon)x_n^\gamma \geq u_n \text{ along } \{ x' = 0 \}.
\]
Translating the barrier $\psi$ to any $z \in \mathcal{C}_{1/2} \cap \{ x_n = 0 \}$ and repeating the above argument, we find that
\[
(1-\epsilon)x_n^\gamma \geq u_n \text{ in } \mathcal{C}_{1/2} \cap \{ x_n \geq 0\}.
\]

If, on the other hand, the second inequality holds, up to dividing by $c_0$ (so that $c_0 = 1$), consider 
\[
\psi := (1+\epsilon)x_n^\gamma - c_1 \epsilon |x'|^2 + \epsilon x_n^\kappa.
\]
An analogous argument proves that
\[
(1+\epsilon)x_n^\gamma\leq u_n \text{ in } \mathcal{C}_{1/2} \cap \{ x_n \geq 0\}.
\] 

In summary, \eqref{eqn: osc decay at bdry} holds in Case 1.

\medskip
{\bf Case 2: $\alpha = \beta$.}
Setting $\kappa = 2$ from the start, and following the same line of reasoning proves this case.

\medskip
{\bf Case 3: $\alpha > \beta$.}
By duality, consider $v$ (the Legendre transform of $u$).
Reversing the roles of $\alpha$ and $\beta$, and applying the arguments of Case 1 and 2 proves this case.
\end{proof}

Iterating Proposition~\ref{prop: osc decay at bdry} (rescaling $\mathcal{C}_{1/2}$ to $\mathcal{C}_1$ leaves things unchanged), we find that $u_n/x_n^\gamma$ is H\"older continuous at the origin.
Translating this argument to other boundary points yields that $u_n/x_n^\gamma$ is locally H\"older continuous up to $\{x_n = 0 \}$.
In particular, we have the following corollary.

\begin{corollary}
\label{cor: C2eta in normal}
Let $u$ be convex and satisfy \eqref{eqn: global OT flat}.
A constant $\chi \in (0,1)$ exists, depending only on $\alpha$, $\beta$, and $n$, such that
\[
\bigg[ \frac{u_n}{x_n^\gamma}\bigg]_{C^{0,\chi}(\mathcal{C}_1 \cap \{ x_n \geq 0\})} \leq C(\alpha,\beta,n,\|u_n/x_n^\gamma\|_{L^\infty(\mathcal{C}_1 \cap \{ x_n \geq 0\})}).
\]
\end{corollary}

From Corollary~\ref{cor: C2eta in normal}, we deduce Theorem~\ref{thm: liouville}.

\begin{proof}[Proof of Theorem~\ref{thm: liouville}]
By Proposition~\ref{prop: pogo tang}, duality, and Lemma~\ref{lem: pure normal comp to 1}, there exists an $M > 0$ such that
\[
\frac{1}{M}\Id' \leq D^2_{x'}u \leq M \Id' \text{ and } \frac{1}{M} \leq \frac{u_n}{x_n^\gamma} \leq M \text{ in } B'_1 \times \{ 0 \}. 
\]
After subtracting off the tangent plane to $u$ at any point in $\mathcal{C}_1 \cap \{ x_n = 0 \}$ and a translation, if we can show that 
\begin{equation}
\label{eqn: sections round Liouville}
\mathcal{C}_{1/R} \cap \{ x_n \geq 0\} \subset S = \{ \psi < 1 \} \subset \mathcal{C}_{R} \cap \{ x_n \geq 0\} \text{ with } \psi := u_t,
\end{equation}
for some $R = R(M) > 0$, then (since $u_n = 0$ on $\{ x_n = 0\}$)
\[
\bigg\| \frac{\partial_n u_t}{x_n^\gamma} \bigg\|_{L^\infty(\mathcal{C}_1 \cap \{x_n \geq 0\})} = \bigg\| \frac{\psi_n}{x_n^\gamma} \bigg\|_{L^\infty(\mathcal{C}_1 \cap \{x_n \geq 0\})}\leq C(M).
\]
(Recall that $u_t$ is the rescaling of $u$ defined in \eqref{eqn: rescalings}.)
Therefore, by Corollary~\ref{cor: C2eta in normal} and scaling,
\[
\bigg[ \frac{u_n}{x_n^\gamma}\bigg]_{C^{0,\chi}(\mathcal{C}_t \cap \{ x_n \geq 0\})} \leq \frac{C(M)}{\min\{ t^{\frac{\chi}{2}},  t^{\frac{\chi}{1+\gamma}} \}} \to 0 \text{ as } t \to \infty.
\]
Hence, $u_n/x_n^\gamma$ is constant in $\{ x_n \geq 0\}$.
In turn, $\det D^2_{x'}u$ is constant in $\{ x_n > 0 \}$. 
So, by J\"{o}rgens, Calabi, and Pogorelov's Liouville theorem (\cite{Calabi,Jorgens,P}),
\[
u(x) = P(x') + p x_n^{1+\gamma}
\]
for some uniformly convex quadratic polynomial $P$, proving the theorem.

Now we prove \eqref{eqn: sections round Liouville}.
By the arguments of Lemma~\ref{lem: pure normal comp to 1}, we find that the linear map $A$ that normalizes the pair $(\psi, S)$ is such that
\[
A = \diag(A',a) \text{ and } \frac{1}{C} \leq (\det A)^2a^{\alpha + \beta} \leq C,
\]
for some $C > 0$, depending only on $\alpha$, $\beta$, and $n$.
Let
\[
\tilde{\psi}(x) := \psi(A^{-1}x),
\]
so that 
\begin{equation}
\label{eqn: normalized section and image Liouville}
B_1(0) \subset \tilde{S} := A(S) \subset B_{n^{3/2}}(0) \text{ and } \nabla \tilde{\psi}(\tilde{S}) \subset B_{1/r}(0),
\end{equation}
for some $r > 0$, depending only on $\alpha$, $\beta$, and $n$.
The inclusion concerning the gradient of $\tilde{\psi}$ follows from Corollary~\ref{cor: sections are ellipsoids} and Corollary~\ref{cor: equivalence of sections}.
So, by Proposition~\ref{prop: pogo tang} and duality,
\[
\frac{1}{\tilde{M}}\Id' \leq D^2_{x'}\tilde{\psi}(0) \leq \tilde{M} \Id'.
\]
Furthermore, by Lemma~\ref{lem: pure normal comp to 1},
\[
\frac{1}{C(\tilde{M})} \leq \frac{\tilde{\psi}_n}{x_n^\gamma}(0) \leq C(\tilde{M}).
\]
Here and above $\tilde{M} > 0$ and $C(\tilde{M}) > 0$ depend only on $\alpha$, $\beta$, and $n$.
It follows that
\[
\frac{1}{M\tilde{M}}\Id' \leq (A')^tA' \leq M\tilde{M} \Id' \text{ and } \frac{1}{MC(\tilde{M})} \leq a \leq MC(\tilde{M}).
\]
These inequalities, by \cite[Lemma~A.4]{Fbook}, imply that 
\[
| A |, |A^{-1}| \leq C(M).
\]
In turn, from \eqref{eqn: normalized section and image Liouville} we deduce \eqref{eqn: sections round Liouville}, as desired.
\end{proof}

%~~~~~~~~~~~~~~~~~~~~~~~~~~~~~~~~~~~~~~~~~~~~~~~~~~~~~~~%
\section{Flat Implies Smooth}
\label{sec: Flat Implies Smooth}

In this section, we prove a pointwise ``flat implies smooth'' result.
Before doing so, we reintroduce and introduce some notation essential to the statements and proofs of this section.

Recall, 
\[
\mathcal{C}_r(z) = B'_{r^{1/2}}(z') \times (z_n - r^\frac{1}{1+\gamma}, z_n + r^\frac{1}{1+\gamma}) \text{ and } \mathcal{C}_r = \mathcal{C}_r(0).
\] 

Let
\[
\overline{\lambda} := \lambda + \frac{2}{1+\gamma} = 1 + \lambda + \frac{1-\gamma}{1+\gamma} \text{ and } \underline{\lambda} := \lambda + \frac{2\gamma}{1+\gamma} = 1 + \lambda - \frac{1-\gamma}{1+\gamma}.
\]
Also, let $X$ be an open set whose boundary in $\mathcal{C}_1$ is defined by be a nonnegative function $\Gamma_X = \Gamma_X(x')$ on $\R^{n-1}$:
\[
X \cap \mathcal{C}_1= \{ x_n > \Gamma_X(x') \} \cap \mathcal{C}_1 \text{ and } \partial X \cap \mathcal{C}_1 =  \{ x_n = \Gamma_X(x') \} \cap \mathcal{C}_1.
\] 

Set
\[
U(x) := \frac{|x'|^2}{2} + \gamma^{\frac{\beta}{1+\beta}}\frac{x_n^{1+\gamma}}{(1+\gamma)\gamma}.
\]
Finally, define
\[
\mathcal{C}^\ast_r(z) := B'_{r^{1/2}}(z') \times (z_n - r^\frac{\gamma}{1+\gamma}, z_n + r^\frac{\gamma}{1+\gamma}) \text{ and } \mathcal{C}^\ast_r := \mathcal{C}^\ast_r(0).
\]

\begin{proposition}
\label{prop: C2lambda pointwise}
Assume that $X$ is convex with $0 \in \partial X$.
Let $Y \subset \{ y_n > 0 \}$ be an open set with $0 \in \partial Y$, and assume further that
\begin{equation}
\label{eqn: bdrys geometrically flat}
0 \leq x_n \leq \delta\epsilon|x'|^{\overline{\lambda}} \text{ on } \partial X \cap \mathcal{C}_2 \text{ and } 0 \leq y_n \leq \delta\epsilon |y'|^{\underline{\lambda}} \text{ on } \partial Y \cap \mathcal{C}_{1/\rho}^\ast.
\end{equation}
Suppose that $u \in C^1(\overline{X} \cap \mathcal{C}_2)$ is a convex function such that
\begin{equation}
\label{eqn: tanget at 0 is 0}
u(0) = 0 = |\nabla u(0)|,
\end{equation}
$u_n \geq 0$, and
\[
\nabla u(\partial X \cap \mathcal{C}_1) \subset \partial Y \cap \mathcal{C}_{1/\rho}^\ast.
\]
In $X \cap \mathcal{C}_1$, assume that
\[
\frac{1 - \delta\epsilon|x|^\mu}{1 + \delta\epsilon|\nabla u|^\omega }\frac{(x_n - 2\delta\epsilon|x'|^{\overline{\lambda}})_+^\alpha}{u_n^\beta} \leq \det D^2 u \leq \frac{1 + \delta\epsilon|x|^\mu }{1- \delta\epsilon|\nabla u|^\omega }\frac{x_n^\alpha}{(u_n - 2 \delta\epsilon|\nabla_{x'} u|^{\underline{\lambda}})_+^\beta}.
\]
Furthermore, suppose that
\[
|u - U| \leq \epsilon \text{ in } X \cap \mathcal{C}_2.
\]
If $\delta, \epsilon, \rho > 0$ are sufficiently small, then
\[
|u(Rx) - U| \leq C\epsilon(|x'|^2 + x_n^{1+\gamma})^{1 + \frac{\lambda}{2}}
\]
for some $R = \diag(R',r_n)$.
\end{proposition}

Our proposition will follow from an iteration of an improvement of flatness lemma.
In order to state it, however, we need to define one more object:
\[
U_{\tau'}(x) := U(x) + \tau' \cdot x' \text{ so that } U_0 = U.
\]
Also, recall,
\[
D_h := \diag(h^{\frac{1}{2}}\Id', h^{\frac{1}{1+\gamma}}).
\]

\begin{lemma}
\label{lem: iteration step}
Let $0 \in \partial X$ and assume that
\begin{equation}
\label{eqn: boundaries are flat}
0 \leq x_n \leq  \delta \epsilon \text{ on } \partial X \cap \mathcal{C}_1.
\end{equation}
Furthermore, let $u \in C^1(\overline{X} \cap \mathcal{C}_1)$ be convex and such that
\begin{equation}
\label{eqn: u at 0 is 0}
u(0) = 0,
\end{equation}
$u_n \geq 0$, and
\begin{equation}
\label{eqn: bdry to bdry}
0 \leq u_n \leq  \delta \epsilon \text{ on } \partial X \cap \mathcal{C}_1.
\end{equation}
Suppose that
\[
\frac{1 - \delta\epsilon }{1 + \delta\epsilon }\frac{(x_n - \delta \epsilon)_+^\alpha}{u_n^\beta} \leq \det D^2 u \leq \frac{1 + \delta\epsilon }{1- \delta\epsilon }\frac{x_n^\alpha}{(u_n - \delta \epsilon)_+^\beta}  \text{ in } X \cap \mathcal{C}_1.
\]
For any $0 < \lambda < 1$, there exist constants $\delta_0 > 0$, $\epsilon_0 > 0$, and $h_0 > 0$ such that the following holds: if
\[
| u - U_{\tau'} | \leq \epsilon \text{ in } X \cap \mathcal{C}_1,
\]
then
\[
|\tilde{u} - U_{\tilde{\tau}'} | \leq h_0^{\frac{\lambda}{2}} \epsilon \text{ in } \tilde{X} \cap \mathcal{C}_2
\]
provided $0 < \delta \leq \delta_0$ and $0 < \epsilon \leq \epsilon_0$.
Here
\[
\tilde{u}(x) := \frac{u(QD_{h_0}x)}{h_0},\, \tilde{\tau}' := \frac{(Q'D_{h_0}')^t(\tau' + q')}{h_0},\, \text{and } \tilde{X} := (QD_{h_0})^{-1}X
\]
for some $Q$ and $q'$ such that
\[
Q = \diag(Q',q_n) \text{ and } |Q - \Id|, |q'| \leq C_0 \epsilon
\]
and
\begin{equation}
\label{eqn: det almost 1}
q_n^{\alpha+\beta}(\det Q)^2 = 1.
\end{equation}
\end{lemma}

To prove our improvement of flatness lemma, we approximate $u - U_{\tau'}$ by a solution to the Grushin type equation with singular drift $L w = 0$ where $L$ is defined by
\[
Lw := \gamma^{\frac{\beta}{1+\beta}} x_n^{\gamma-1}\Delta_{x'}w + w_{nn} + \beta\gamma\frac{w_n}{x_n},
\] 
which with the Neumann condition $w_n = 0$ on $\{ x_n = 0 \}$ has a rather nice regularity theory (see \cite{DS}).

\begin{proof}
Let
\[
u_\epsilon := \frac{u - U_{\tau'} }{\epsilon}.
\]

\medskip
{\bf Step 1:} We show that $u_\epsilon$ is well-approximated near the origin by a solution $w$ to the linearized equation
\begin{equation}
\label{eqn: linearized equation}
\begin{cases}
Lw = 0 &\text{in }  \mathcal{C}_{1/4} \cap \{x_n > 0 \} \\
w_n = 0 &\text{on } \mathcal{C}_{1/4} \cap \{ x_n = 0 \}
\end{cases}
\end{equation}
in the viscosity sense.
In other words, for any $\eta > 0$, a solution $w$ to \eqref{eqn: linearized equation} in $\mathcal{C}_{1/4} \cap \{ x_n > 0\}$ exists such that
\[
|u_\epsilon - w| \leq \eta \text{ in } X \cap \mathcal{C}_{1/4}
\]
provided $\delta_0, \epsilon_0 > 0$ are sufficiently small, depending on $\eta$, $\rho$, $n$, $\alpha$, and $\beta$.

\smallskip
{\it Step 1.1:}
First, we derive the linearized equation.

To this end, by convexity, observe that if $x + t\e \in X \cap \mathcal{C}_1$, then
\[
(\nabla u(x) - \nabla U_{\tau'}(x) )\cdot t\e \leq 2\epsilon  + (\nabla U_{\tau'}(x + t\e) - \nabla U_{\tau'}(x) )\cdot t\e;
\]
in particular,
\[
|u_n(x) - U_n(x)| \leq 2\epsilon^{\frac{1}{2}} + C_\gamma\max\{\epsilon^{\frac{\gamma}{2}},\epsilon^{\frac{1}{2}}\}.
\]
Consequently, $u_n \to \partial_n U_{\tau'} = U_n$ uniformly in compact subsets of $X \cap \mathcal{C}_1$ as $\epsilon \to 0$.
Moreover,
\[
\det D^2 u - \det D^2 U_{\tau'} = \trace (A_\epsilon D^2 (u - U_{\tau'})),
\]
with
\[
A_\epsilon := \int_0^1 \cof ( D^2 U + t( D^2 u - D^2 U)) \, \d t.
\]
Since $\det^{1/n}$ is concave on symmetric, positive semi-definite $n \times n$ matrices,
\[
(\det A_\epsilon)^{1/n} \geq \int_{0}^1 (t(\det D^2u)^{1/n} + (1-t)(\det D^2U)^{1/n})\, \d t.
\]
Therefore, $(\det A_\epsilon)^{1/n}$ is strictly positive and bounded on compact subsets of $X \cap \mathcal{C}_1$.
Furthermore,
\[
\begin{split}
\frac{x_n^\alpha}{(u_n - \delta \epsilon)_+^\beta} - \frac{x_n^\alpha}{U_n^\beta} &= \frac{x_n^\alpha}{u_n^\beta} - \frac{x_n^\alpha}{U_n^\beta} + \frac{x_n^\alpha}{(u_n - \delta \epsilon)_+^\beta} - \frac{x_n^\alpha}{u_n^\beta} 
\\
&= - x_n^\alpha \int_0^1 \beta \frac{u_n - U_n}{(U_n + t(u_n - U_n))^{1+\beta}} \, \d t + \frac{x_n^\alpha}{(u_n - \delta \epsilon)_+^\beta} - \frac{x_n^\alpha}{u_n^\beta}.
\end{split}
\]
And so,
\[
\frac{1 + \delta\epsilon }{1-  \delta\epsilon }\frac{x_n^\alpha}{(u_n - \delta \epsilon)_+^\beta} - \frac{x_n^\alpha}{U_n^\beta} \leq b^+_{\epsilon} \partial_n (u - U_{\tau'}) + c^+_{\delta}\epsilon
\]
with
\[
b^+_{\epsilon} \to -\beta \frac{x_n^\alpha}{U_n^{1+\beta}} \text{ as } \epsilon \to 0 \text{ and } c^+_{\delta} \to 0 \text{ as } \delta \to 0
\]
locally uniformly in $X \cap \mathcal{C}_1$.
Similarly,
\[
\frac{1 -  \delta\epsilon }{1 +  \delta\epsilon }\frac{(x_n - \delta \epsilon)_+^\alpha}{u_n^\beta} - \frac{x_n^\alpha}{U_n^\beta} \geq b^-_{\epsilon} \partial_n (u - U_{\tau'}) + c^-_{\delta}\epsilon
\]
with
\[
b^-_{\epsilon} \to -\beta \frac{x_n^\alpha}{U_n^{1+\beta}} \text{ and } c^-_{\delta} \to 0
\]
uniformly in compact subsets of $X \cap \mathcal{C}_1$ as $\epsilon, \delta \to 0$, respectively.
In turn, by the ABP estimate and Schauder theory, $D^2 u \to D^2U_{\tau'} = D^2 U$ locally uniformly in $X \cap \mathcal{C}_1$ as $\epsilon \to 0$.
Hence,
\[
A_\epsilon \to \cof D^2 U = \diag(\gamma^{\frac{\beta}{1+\beta}}x_n^{\gamma-1} \Id',1)
\]
as $\epsilon \to 0$ uniformly in compact subsets of $X \cap \mathcal{C}_1$.

\smallskip
{\it Step 1.2:}
Second, we show that $u_\epsilon$ is uniformly H\"{o}lder continuous in an appropriate sense up to $\mathcal{C}_{1/4} \cap \{ x_n = 0 \}$ as $\delta, \epsilon \to 0$.

To start, we claim that two small constants $c_0 > 0$ and $\delta_0 > 0$ exist such that the following holds: for all $\delta \leq \delta_0$, if
\[
\osc_{X \cap \mathcal{C}_1}  u_\epsilon \leq 2,
\]
then
\[
\osc_{X \cap \mathcal{C}_{1/2}} u_\epsilon \leq 2(1-c_0).
\]

We prove this claim with a barrier argument.
For every $z \in \partial X \cap \mathcal{C}_{1/2}$, define
\[
\phi_{z}(x) := -1 + c_2\bigg(1 + 2C_1\gamma^{\frac{\beta}{1+\beta}}\frac{x_n^{1+\gamma}}{(1+\gamma)\gamma} + x_n - C_1\frac{|x'-z'|^2}{2}\bigg)
\]
and
\[
\phi^{z}(x) := 1 - c_2\bigg(1 + 2C_1\gamma^{\frac{\beta}{1+\beta}}\frac{(x_n - \delta \epsilon)_+^{1+\gamma}}{(1+\gamma)\gamma} - C_1\frac{|x'-z'|^2}{2}\bigg),
\]
for $C_1 \gg 1$ and $c_2 \ll 1$ with $c_2C_1 \ll 1$ to be chosen (uniformly in $z$).
Also, let
\[
F_{\delta,\epsilon}^{+}(D^2 \psi, \nabla \psi , x) := (1- \delta\epsilon )(\psi_n - \delta \epsilon)_+^\beta \det D^2 \psi - (1 +  \delta\epsilon)x_n^\alpha
\]
and
\[
F_{\delta,\epsilon}^{-}(D^2 \psi, \nabla \psi , x) := (1 + \delta\epsilon )\psi_n^\beta\det D^2 \psi  - (1 - \delta\epsilon )(x_n - \delta \epsilon)_+^\alpha. 
\]
Finally, define
\[
w_{z} := U_{\tau'} + \epsilon \phi_{z}
\]
and
\[
w^{z} := \frac{|x'|^2}{2} + \gamma^{\frac{\beta}{1+\beta}}\frac{(x_n - \delta \epsilon)_+^{1+\gamma}}{(1+\gamma)\gamma} + \epsilon \phi^{z} + \tau' \cdot x'.
\]

First, we show that if $\delta > 0$ is small enough (and $\epsilon < 1$), then 
\[
F_{\delta,\epsilon}^+(D^2 w_{z}, \nabla w_{z} , x) \geq 0 \geq F_{\delta,\epsilon}^+(D^2u, \nabla u , x) \text{ in } X \cap \mathcal{C}_1
\]
and
\[
F_{\delta,\epsilon}^-(D^2 w^{z}, \nabla w^{z} , x) \leq 0 \leq F_{\delta,\epsilon}^-(D^2u, \nabla u , x) \text{ in } X \cap \mathcal{C}_1.
\]
% {\color{teal}
% Note
% \[
% D^2 w_{z} = \diag( (1 - c_2C_1\epsilon)\Id', (1 + 2c_2C_1\epsilon)\gamma^{\frac{\beta}{1+\beta}}x_n^{\gamma-1} ).
% \]
% Also,
% \[
% \partial_n w_{z} = \gamma^{-\frac{1}{1+\beta}}x_n^{\gamma}(1 + 2c_2C_1\epsilon) + c_2A\epsilon. 
% \]
% So
% \[
% \begin{split}
% F_{\delta,\epsilon}^+(D^2 w_{z}, \nabla w_{z} , x) &+ (1 +  \delta\epsilon)x_n^\alpha = 
% \\
% &(1-\delta \epsilon)(\gamma^{-\frac{1}{1+\beta}}x_n^{\gamma}(1 + 2c_2C_1\epsilon) + (c_2A-\delta)\epsilon)^\beta \times
% \\
% &\times (1 - c_2 C_1 \epsilon)^{n-1}(1 + c_2 2C_1 \epsilon)\gamma^{\frac{\beta}{1+\beta}}x_n^{\gamma-1}. 
% \end{split}
% \]
% Note that
% \[
% RHS = (1-\delta \epsilon)\{ \text{factor} \}x_n^\alpha.
% \]
% where
% \[
% \{ \text{factor}  \} = (1 + A_1\epsilon)^\beta(1 - c_2 C_1 \epsilon)^{n-1}(1 + c_2 2C_1 \epsilon)
% \]
% with
% \[
% A_1 = 2c_2C_1+ \gamma^{\frac{1}{1+\beta}}x_n^{-\gamma}(c_2A-\delta).
% \]
% If $\beta \geq 1$, then, by the convexity, $(a+b)^\beta \geq a^\beta(1 + \beta ba^{-1})$; so
% \[
% \begin{split}
% &\geq (1 - \delta \epsilon)(1 - c_2 C_1 \epsilon)(1 + c_2 2C_1 \epsilon)^{1+\beta}x_n^\alpha(1 + \beta a^{-1}(c_2A-\delta)\epsilon).
% \\
% &\geq (1 - \delta \epsilon)(1 - c_2 C_1 \epsilon)(1 + c_2 2C_1 \epsilon)^{1+\beta}x_n^\alpha(1 + 2^{-1}\beta\gamma^{\frac{1}{1+\beta}}(c_2A-\delta)\epsilon)
% \\
% &\geq (1 - \delta \epsilon)(1 - c_2 C_1 \epsilon)(1 + c_2 2C_1 \epsilon)(1 + C_A\epsilon) x_n^\alpha
% \end{split}
% \]
% with
% \[
% C_A = \frac{\beta\gamma^{\frac{1}{1+\beta}}(c_2A-\delta)}{2}.
% \]
% If $\beta < 1$, then
% }
Indeed, if $\delta \leq c_2$ (recalling that $c_2C_1 \ll 1$ and $C_1 \gg 1$), then
\[
F_{\delta,\epsilon}^+(D^2 w_{z}, \nabla w_{z} , x) \geq [(1 - \delta \epsilon)(1 - c_2 C_1 \epsilon)(1 + c_2 2C_1 \epsilon) - (1 + \delta \epsilon)]x_n^\alpha \geq 0
\]
% {\color{teal}
% We need
% \[
% 0 \leq (1 - \delta \epsilon)(1 - c_2 C_1 \epsilon)(1 + c_2 2C_1 \epsilon) - (1 + \delta \epsilon),
% \]
% ie
% \[
% (1 + \delta \epsilon)\leq (1 - \delta \epsilon)(1 - c_2 C_1 \epsilon)(1 + c_2 2C_1 \epsilon),
% \]
% ie
% \[
% \begin{split}
% \frac{1 + \delta \epsilon}{1 - \delta \epsilon} &\leq (1 - c_2 C_1 \epsilon)(1 + c_2 2C_1 \epsilon) 
% \\
% &=1 - c_2C_1\epsilon + c_2 2C_1\epsilon - c_2^2C_12C_1\epsilon^2 
% \\
% &= 1 + (2C_1- C_1)c_2\epsilon - C_12C_1(c_2\epsilon)^2
% \\
% &= 1 + C_1c_2\epsilon - 2(C_1c_2\epsilon)^2.
% \end{split}
% \]
% If $\delta\epsilon \leq 1/2$, then
% \[
% \frac{1 + \delta \epsilon}{1 - \delta \epsilon}  \leq 1 + 4\delta\epsilon.
% \]
% So it suffices to prove that
% \[
% 1 + 4\delta\epsilon \leq 1 + C_1c_2\epsilon - 2(C_1c_2\epsilon)^2.
% \]
% Similarly, since $1 + t - 2t^2 \geq 1 + \frac{t}{2}$ if $t \leq \frac{1}{4}$, it suffices to prove that
% \[
% 1 + 4\delta\epsilon \leq 1 + \frac{C_1c_2}{2}\epsilon,
% \]
% provided $C_1c_2\epsilon \leq 1/4$.
% In other words, we need
% \[
% \delta \leq \frac{C_1c_2}{8}.
% \]
% In summary,
% \[
% \delta\epsilon \leq \frac{1}{2},\, \epsilon \leq \frac{1}{4C_1c_2},\, \delta \leq \frac{C_1c_2}{8}.
% \]
% Since $c_2C_1 \ll 1$ and $\delta, \epsilon < 1$, this reduces to 
% \[
% \delta \leq \frac{C_1c_2}{8},
% \]
% which holds if $\delta \leq c_2$, as $C_1 \gg 1$.
% }
and
\[
F_{\delta,\epsilon}^-(D^2 w^{z}, \nabla w^{z} , x) \leq  [(1 + \delta \epsilon)(1 + c_2 C_1 \epsilon)(1 - c_2 2C_1 \epsilon) - (1 - \delta \epsilon) ](x_n - \delta \epsilon)_+^\alpha \leq 0.
\]
% {\color{teal}
% We need
% \[
% 0 \geq (1 + \delta \epsilon)(1 + c_2 C_1 \epsilon)(1 - c_2 2C_1 \epsilon) - (1 - \delta \epsilon).
% \]
% ie
% \[
% \frac{1 - \delta \epsilon}{1 + \delta \epsilon} \geq (1 + c_2 C_1 \epsilon)(1 - c_2 2C_1 \epsilon).
% \]
% Notice that $(1 - t)/(1+t) \geq 1 - 2t$.
% So it sufficies to show that
% \[
% \begin{split}
% 1 - 2\delta \epsilon &\geq (1 + c_2 C_1 \epsilon)(1 - c_2 2C_1 \epsilon)
% \\
% &= 1 + c_2 C_1 \epsilon  - c_2 2C_1 \epsilon  - c_2^2 2C_1^2 \epsilon^2 
% \end{split}
% \]
% ie
% \[
% - 2\delta \epsilon \geq   + c_2 C_1 \epsilon  - c_2 2C_1 \epsilon  - c_2^2 2C_1^2 \epsilon^2 
% \]
% ie
% \[
% \delta \leq   \frac{c_2 C_1}{2}  + (c_2 C_1)^2 \epsilon 
% \]
% So our old choice on $\delta \leq c_2 \leq c_2C_1/8$ works here too.
% }

Second, we address the boundary data.
Note that
\[
\text{either } u_\epsilon(2^{-\frac{1}{1+\gamma}} \e_n) > 0 \text{ or } u_\epsilon(2^{-\frac{1}{1+\gamma}} \e_n) \leq 0.
\]
In the first case, we prove that
\[
u\geq w_{z} \text{ on } X \cap \partial \mathcal{C}_{1/8}(z)
\text{ and }
\partial_n w_{z} \geq \partial_n u \text{ on } \partial X \cap \mathcal{C}_{1/8}(z).
\]
While, in the second case, we prove that
\[
u \leq w^{z} \text{ on } X \cap \partial \mathcal{C}_{1/8}(z)
\text{ and }
\partial_n u \geq \partial_n w^{z} \text{ on } \partial X \cap \mathcal{C}_{1/8}(z).
\]
In the first case, fix some small distance $\bar{d} > \delta\epsilon$, depending only on $\gamma$.
If $C_1$ is sufficiently large, depending only on $\gamma$, then $\phi_{z} \leq -1$ on $\partial \mathcal{C}_{1/8}(z) \cap \{ x_n \leq \bar{d}  \}$.
Hence, $w_z \leq u$ here.
% {\color{teal}
% For this, since $\phi_z$ in increasing in $x_n$, we need
% \[
% \begin{split}
% -1 &\geq \phi_z(z' + \tfrac{1}{8}\e',d) 
% \\
% &= -1 + c_2(1 + 2C_1C_\gamma d^{1+\gamma} + Ad - \tfrac{C_1}{2^7}).
% \end{split}
% \]
% In other words,
% \[
% 0 \geq c_2(1 + 2C_1C_\gamma d^{1+\gamma} + Ad - \frac{C_1}{2^7});
% \]
% i.e.,
% \[
% C_1 \geq \frac{2^7(1 + d)}{1 - 2^8C_\gamma d^{1+\gamma}}.
% \]
% }
By the Harnack inequality, choosing $c_2 > 0$ small, we can ensure that $w_z \leq u$ on the remainder of $X \cap \partial \mathcal{C}_{1/8}(z)$.
Also, by \eqref{eqn: bdry to bdry}, once again if $\delta \leq c_2$,
\[
\partial_n w_{z} = U_n + \epsilon c_2 2C_1U_n + \epsilon c_2 \geq \epsilon c_2 \geq \partial_n u \text{ on } \partial X \cap \mathcal{C}_{1/8}(z).
\]
In the second case, a similar Harnack inequality argument, yields the required inequality along $X \cap \partial \mathcal{C}_{1/8}(z)$.
Also, by \eqref{eqn: boundaries are flat},
\[
\partial_n w^{z} =  \gamma^{-\frac{1}{1+\beta}}(1 - \epsilon c_2 2C_1)(x_n - \delta \epsilon)_+^\gamma = 0 \leq \partial_n u \text{ on } \partial X \cap \mathcal{C}_{1/8}(z).
\]

By the maximum principle then, one of the two inequalities
\[
w_{z} \leq u \leq w^{z} \text{ in } X \cap \mathcal{C}_{1/8}(z)
\]
holds for all $z \in \partial X \cap \mathcal{C}_{1/2}$, from which the claim follows with $c_0 = c_2/2$.

Iterating the claim, we find that if $\epsilon < 1/2^{k-1}$ and $\delta_0 > 0$ is small, then
\[
\osc_{X \cap \mathcal{C}_{2^{-k}}} u_\epsilon \leq 2(1-c_0)^k.
\]
% {\color{teal}
% Let
% \[
% u_1(x) := \frac{u(D_{h_1}x)}{h_1} \text{ and } \epsilon_1 := 2 \epsilon \text{ and } h_1 := \frac{1}{2}(1-c_0) \text{ and } \tau_1' := h_1^{-1/2} \tau'.
% \]
% Then on $\partial X_1$,
% \[
% 0 \leq x_n \leq \frac{1}{2}h_1^{-\frac{1}{1+\gamma}}\delta \epsilon_1
% \]
% and
% \[
% 0 \leq \partial_n u_1 \leq \frac{1}{2}h_1^{-\frac{\gamma}{1+\gamma}}\delta \epsilon_1
% \]
% and
% \[
% \osc_{\mathcal{C}_1} (u_1 - U_{\tau'_1}) \leq \epsilon_1
% \]
% and
% \[
% \frac{1 - \delta\epsilon }{1 + \delta\epsilon }\frac{(x_n - \frac{1}{2}h_1^{-\frac{1}{1+\gamma}}\delta\epsilon_1)_+^\alpha}{(\partial_n u_1)^\beta} \leq \det D^2 u_1 \leq \frac{1 + \delta\epsilon }{1- \delta\epsilon }\frac{x_n^\alpha}{(\partial_n u_1  - \frac{1}{2}h_1^{-\frac{\gamma}{1+\gamma}}\delta\epsilon_1)_+^\beta} 
% \]
% So take 
% \[
% \delta_1 = \max \{ \frac{1}{2}h_1^{-\frac{1}{1+\gamma}} \delta, \frac{1}{2}h_1^{-\frac{\gamma}{1+\gamma}}\delta, \delta \} = \delta
% \]
% if $c_0 \ll 1$ depending only on $\gamma$.
% }
Thus, after translating the above argument to any point $x_0 \in \partial X \cap \mathcal{C}_{1/4}$, we see that $u_\epsilon$ converges uniformly in $\mathcal{C}_{1/4} \cap \{ x_n \geq 0 \}$, as $\epsilon$ and $\delta$ tend to zero, to some function $w$ that solves
\[
Lw = \gamma^{\frac{\beta}{1+\beta}} x_n^{\gamma-1}\Delta_{x'}w + w_{nn} + \beta\gamma\frac{w_n}{x_n} = 0 \text{ in } \{ x_n > 0 \} \cap  \mathcal{C}_{1/4}
\] 
in the viscosity sense (and so, classically by elliptic theory).

\smallskip
{\it Step 1.3:}
Third, we show that the Neumann condition $w_n = 0$ is satisfied in the viscosity sense as defined in \cite[Definition 7.1]{DS} on $\mathcal{C}_{1/4} \cap \{ x_n = 0 \}$.

If $\beta \gamma \geq 1$, then $w_n = 0$ on $\{ x_n = 0 \}$ in the viscosity sense since $|w| \leq 1$. 
This bound is a consequence of Step 1.2.

When $\beta \gamma < 1$, we show that $w$ can neither be touched from above at any point on $\mathcal{C}_{1/4} \cap \{x_n = 0\}$ by any test function of the form
\[
\frac{A}{2}|x'-z'|^2 + B + 2p x_n^{1-\beta\gamma}
\]
where $z' \in \R^{n-1}$, $A,B \in \R$, and
\[
p < 0
\]
(making $w$ a viscosity subsolution) nor be touched from below at any point on $\mathcal{C}_{1/4} \cap \{x_n = 0\}$ by any test function of the form
\[
\frac{A}{2}|x'-z'|^2 + B + 2p x_n^{1-\beta\gamma}
\]
where $z' \in \R^{n-1}$, $A,B \in \R$, and 
\[
p > 0
\]
(making $w$ a viscosity supersolution).

Suppose, to the contrary, that $w$ can be touched from below at some $x_0 = (x_0',0) \in \mathcal{C}_{1/4} \cap \{x_n = 0\}$ by
\[
\frac{A}{2}|x'-z'|^2 + B + 2p x_n^{1-\beta\gamma}
\]
for some $z' \in \R^{n-1}$, $A,B \in \R$, and $p > 0$.
Since $\beta\gamma < 1$, we can touch $w$ at $x_0$ from below strictly by
\[
\phi(x) := \frac{A}{2}|x'-z'|^2 + B + \frac{\gamma^{\frac{\beta}{1+\beta}}}{(1+\gamma)\gamma} C x_n^{1+\gamma}+ p x_n
\]
with any $C \in \R$.
Since $u_\epsilon \to w$ uniformly as $\epsilon, \delta \to 0$,
\[
\Phi := U_{\tau'} + \epsilon(\phi + c_{\epsilon,\delta})
\]
touches $u$ from below strictly at some $x_\epsilon \in X$.
Arguing as in Step 1.2, we find that
\[
F^+_{\delta,\epsilon}(D^2 \Phi(x_\epsilon), \nabla \Phi(x_\epsilon),x_\epsilon) > 0
\]
provided $0 < \delta,\epsilon \ll 1$ and $C \gg 1$ (since $p > 0$).
But then
\[
0 \geq F^+_{\delta,\epsilon}(D^2 u(x_\epsilon), \nabla u(x_\epsilon) ,x_\epsilon) \geq F^+_{\delta,\epsilon}(D^2 \Phi(x_\epsilon), \nabla \Phi(x_\epsilon),x_\epsilon) > 0,
\]
which is impossible.
(The first inequality is an assumption on $u$, and the middle inequality holds since $\Phi$ touches $u$ from below.)
On the other hand, suppose, to the contrary, that $w$ can be touched from above at some $x_0 = (x_0',0) \in \mathcal{C}_{1/4} \cap \{x_n = 0\}$ by
\[
\frac{A}{2}|x'-z'|^2 + B + 2p x_n^{1-\beta\gamma}
\]
for some $z' \in \R^{n-1}$, $A,B \in \R$, and $p < 0$.
Since $\beta\gamma < 1$, we can touch $w$ at $x_0$ from above strictly by
\[
\phi(x) := \frac{A}{2}|x'-z'|^2 + B + \frac{\gamma^{\frac{\beta}{1+\beta}}}{(1+\gamma)\gamma} C(x_n - \delta \epsilon)_+^{1+\gamma} + p x_n
\]
with any $C \in \R$.
Since $u_\epsilon \to w$ uniformly as $\delta, \epsilon \to 0$,
\[
\Phi := \frac{|x'|^2}{2} + \gamma^{\frac{\beta}{1+\beta}}\frac{(x_n - \delta \epsilon)_+^{1+\gamma}}{(1+\gamma)\gamma} + \tau' \cdot x' + \epsilon(\phi + c_{\epsilon,\delta})
\]
touches $u$ from below strictly at some $x_\epsilon \in X$.
Arguing as in Step 1.2, we find that
\[
F^-_{\delta,\epsilon}(D^2 \Phi(x_\epsilon), \nabla \Phi(x_\epsilon),x_\epsilon) < 0
\]
provided $0 < \delta,\epsilon \ll 1$ and $C \ll -1$ (since $p < 0$).
But again this inequality is impossible.

\medskip
{\bf Step 2:}
Now we find the transformation $Q$ and prove \eqref{eqn: det almost 1}.

An application of the arguments of \cite[Section 7.1]{DS} yields that $w \in C^{1 + \gamma}_{\rm loc}(\mathcal{C}_{1/4} \cap \{ x_n \geq 0 \})$.
In particular, $w_n = 0$ is satisfied in the classical sense.
Moreover, $D^k_{x'} w \in C^{1 + \gamma}_{\rm loc}(\mathcal{C}_{1/4} \cap \{ x_n \geq 0 \})$.
Also, since $w$ is the limit of a sequence of functions that vanish at the origin (by \eqref{eqn: u at 0 is 0}),
\[
w(0) = 0.
\]
In turn, by Taylor's theorem and \eqref{eqn: linearized equation},
\[
w(x) = p' \cdot x' + \frac{1}{2}P' x' \cdot x' + C_\gamma p_n x_n^{1+\gamma} + O(|x'|^3 + x_n^{2 + 2\gamma} + |x'|x_n^{1+\gamma})
\]
where
\[
|p'|, |P'|, |p_n| \leq C;
\]
in particular,
\[
p_n = - \frac{\trace P'}{1+\beta} \text{ and } C_\gamma = \frac{\gamma^{\frac{\beta}{1+\beta}}}{(1+\gamma)\gamma}.
\]
It follows that, if $0 < \delta, \epsilon \ll 1$,
\[
\bigg| u - U_\tau -  \epsilon \bigg(p' \cdot x' + \frac{1}{2}P' x' \cdot x' + C_\gamma p_n  x_n^{1+\gamma}\bigg)\bigg| \leq \epsilon\eta + \epsilon C(|x'|^3 + x_n^{2 + 2\gamma} + |x'|x_n^{1+\gamma}).
\]
So
\[
|u( \tilde{Q}'x', \tilde{q}_n x_n) - U(x) - (\tau' + q') \cdot \tilde{Q}'x' | \leq \eta \epsilon+ C h^{\frac{3}{2}}\epsilon + C\epsilon^2 \text{ in } \tilde{Q}^{-1}(X \cap \mathcal{C}_{4h})
\]
with
\[
\tilde{Q} = \diag(\tilde{Q}',\tilde{q}_n) := ((\Id' + \epsilon P')^{-\frac{1}{2}}, (1+\epsilon p_n)^{-\frac{1}{1+\gamma}}) \text{ and } q' := \epsilon p'.
\]
In turn, if $\epsilon, \eta \leq h^{3/2}$,
\[
|u( \tilde{Q}'x', \tilde{q}_n x_n) - (\tau' + q') \cdot \tilde{Q}' x' - U(x)| \leq C h^{\frac{3}{2}} \epsilon \text{ in } (\tilde{Q}^{-1}X) \cap \mathcal{C}_{2h}.
\]
Now notice that
\[
\frac{2 + \alpha + \beta}{1+\gamma} p_n = - \frac{2 + \alpha + \beta}{1+\gamma} \frac{\trace  P'}{1+\beta} = - \trace  P'.
\]
And so,
\[
\begin{split}
\tilde{q}_n^{\alpha+\beta}(\det \tilde{Q})^{2} 
&= (1+\epsilon p_n)^{-\frac{2 + \alpha + \beta}{1+\gamma}} (\det (\Id' + \epsilon P'))^{-1}
\\
&
=  \Big(1 - \epsilon \tfrac{2 + \alpha + \beta}{1+\gamma} p_n + O(\epsilon^2)\Big) \Big(1 -  \epsilon \trace P' + O(\epsilon^2)\Big)
= 1 + O(\epsilon^2).
\end{split}
\]
Thus, we can find a $Q$ that satisfies \eqref{eqn: det almost 1} and is $\epsilon^2$ close to $\tilde{Q}$, i.e.,
\[
|\tilde{Q} - Q| \leq \tilde{C} \epsilon^2.
\]
It follows that
\[
|u( Q'x', q_n x_n) - (\tau' + q') \cdot Q' x' - U(x)| \leq C h^{\frac{3}{2}} \epsilon \text{ in } (Q^{-1}X) \cap \mathcal{C}_{2h},
\]
since $\epsilon \leq h^{3/2}$.
Therefore, taking $h > 0$ sufficiently small (depending on $\lambda < 1$), we find that
\[
\bigg|\frac{1}{h}u( QD_h x) - \frac{1}{h} (\tau' + q') \cdot Q'D_h'x' - U(x) \bigg| \leq \epsilon h^{\frac{\lambda}{2}} \text{ in } ((QD_h)^{-1}X) \cap \mathcal{C}_2,
\]
as desired.
\end{proof}

With Lemma~\ref{lem: iteration step} in hand, we prove our proposition.

\begin{proof}[Proof of Proposition~\ref{prop: C2lambda pointwise}]
First, we claim, by induction, that a sequence of matrices $R_k := \diag(R_k',r_{k,n})$ and vectors $\tau_k' \in \R^{n-1}$ exit such that the rescalings of $u$ at height $h_k = h_0^k$,
\[
u_k(x) := \frac{u(R_k D_k x)}{h_k} \text{ for }  x \in X_k := (R_kD_k)^{-1}X
\]
with
\[
D_k := \diag(h_k^{\frac{1}{2}}\Id', h_k^{\frac{1}{1+\gamma}})
\]
satisfy
\begin{equation}
\label{eqn: induction step close to poly}
|u_k - U_{\tau_k'}| \leq \epsilon_k := \epsilon h_k^{\frac{\lambda}{2}} \text{ in } X_k \cap \mathcal{C}_2
\end{equation}
provided 
\[
\delta = c \delta_0 \text{ and } \epsilon = \epsilon_0
\] 
for some small $c > 0$.
Moreover,
\begin{equation}
\label{eqn: induction step everything small}
r_{k,n}^{\alpha+\beta}(\det R_k)^2 = 1 
\text{ and } |R_k - R_{k-1}| \leq C\epsilon_{k-1}.
\end{equation}

The base case $k = 0$ holds by assumption, with $R_0 = \Id$ and $\tau_0' = 0$.

Now suppose the claim holds for some $k \in \N$.
Note that the second inequality in \eqref{eqn: induction step everything small} implies that
\[
|R_k - \Id| \leq C\epsilon.
\]
From \eqref{eqn: bdrys geometrically flat}, provided that $c > 0$ is sufficiently small, we see that
\[
0 \leq x_n \leq |r_{k,n}^{-1}||R'_k|^{\lambda + \frac{2}{1+\gamma}}\delta \epsilon_k |x'|^{\lambda + \frac{2}{1+\gamma}} \leq \delta_0 \epsilon_k \text{ on } \partial X_k \cap \mathcal{C}_2.
\]
Furthermore, since the segment between $0$ and $\epsilon_k^{1/2} (\tau_k'|\tau_k'|^{-1} + \delta_0\epsilon_k \e_n)$ lives inside $X_k \cap \mathcal{C}_2$ (by the above inequality on the height of $\partial X_k \cap \mathcal{C}_2$ and the convexity of $X$), the function $w_k = u_k +  \tau_k' \cdot x'$ is convex, and the second equality in \eqref{eqn: tanget at 0 is 0}, we deduce that
\[
|\tau_k'| \leq 2\epsilon_k^{\frac{1}{2}} + C_\gamma \max\{\epsilon_k^{\frac{\gamma}{2}}, \epsilon_k^{\frac{1}{2}}\}
\]
(cf., the beginning of Step 1.1 in the proof of Lemma~\ref{lem: iteration step}).
In particular, this shows that $\tau_k' \to 0$ as $k \to \infty$.
A similar argument, but also using that the family of slopes $\tau_k'$ is uniformly bounded, yields the inclusion
\[
\nabla u_k(X_k \cap \mathcal{C}_1) \subset Y_k \cap \mathcal{C}_{1/\rho}^\ast  \text{ with } Y_k := (R_kD_k)^{t}h_k^{-1}Y.
\]
By construction, the boundary of $X_k$ maps to the boundary of $Y_k$.
From \eqref{eqn: bdrys geometrically flat}, we also see that
\[
0 \leq y_n \leq |r_{k,n}||(R'_k)^{-1}|^{\lambda + \frac{2\gamma}{1+\gamma}} \delta\epsilon_k |y'|^{\lambda + \frac{2\gamma}{1+\gamma}} \leq \delta_0\epsilon_k \text{ on } \partial Y_k \cap \mathcal{C}_{1/\rho}^\ast.
\]
In turn, in $X_k \cap \mathcal{C}_1$, $u_n \geq 0$ and, using \eqref{eqn: induction step close to poly},
\[
\frac{1 - \delta_0\epsilon_k}{1 + \delta_0\epsilon_k }\frac{(x_n - \delta_0\epsilon_k)_+^\alpha}{(\partial_n u_k)^\beta} \leq \det D^2 u_k \leq \frac{1 + \delta_0\epsilon_k }{1- \delta_0\epsilon_k }\frac{x_n^\alpha}{(\partial_n u_k - \delta_0\epsilon_k)_+^\beta},
\]
taking $c > 0$ smaller if needed depending on $\rho > 0$.
Therefore, by Lemma~\ref{lem: iteration step},
\[
|\tilde{u}_k - U_{\tilde{\tau}_k'}| \leq \epsilon_kh_0^{\frac{\lambda}{2}} = \epsilon_{k+1} \text{ in } \tilde{X}_k \cap \mathcal{C}_2
\]
where
\[
\tilde{u}_k(x) := \frac{u_k(QD_{h_0}x)}{h_0} \text{ and } \tilde{X}_k := (QD_{h_0})^{-1}X_k
\]
for
\[
Q = \diag(Q',q_n) \text{ with } |Q - \Id|, |q'| \leq C_0 \epsilon_k \text{ and } q_n^{\alpha+\beta}(\det Q)^2 = 1.
\]
In other words, the inductive step holds taking
\[
u_{k+1} = \tilde{u}_k,\, R_{k+1} = R_kQ,\, \text{and } \tau_{k+1}' = \tilde{\tau}_k'.
\]
Indeed,
\[
r_{k+1,n}^{\alpha+\beta}(\det R_{k+1})^2 = q_n^{\alpha+\beta}(\det Q)^2 r_{k,n}^{\alpha+\beta}(\det R_k)^2 = 1,
\]
and 
\[
|R_{k+1} - R_k| \leq |Q - \Id||R_k| \leq 2C_0 \epsilon_k.
\]

Second, we find $R$ and conclude. 
By the inequality in \eqref{eqn: induction step everything small}, we see that $R_k$ converges to some $R$, as $k$ tends to infinity.
In particular,
\[
|R - R_k| \leq C \epsilon_k.
\]
So, after replacing $\epsilon_k$ with $C\epsilon_k$, we can replace $R_k$ by $R$ in \eqref{eqn: induction step close to poly}.
In particular, considering the inductive manner in which each $\tau_k'$ is produced, we have that
\[
|u(Rx) - r_k ' \cdot R' x' - U(x)| \leq C \epsilon h_k^{1+\frac{\lambda}{2}} \text{ in } (R^{-1}X) \cap \mathcal{C}_{h_k}
\]
with
\[
r_k' := \sum_{i=1}^{k} h^{\frac{i-1}{2}}(R_{i-1}')^{-t} q_i'.
\]
Here $q_i'$ is the linear part of the polynomial found at each application of Lemma~\ref{lem: iteration step}.
Hence, $r_k'$ converges to some $r'$, and 
\[
|r' - r_k'| \leq 2\epsilon h_k^{\frac{1}{2} + \frac{\lambda}{2}}.
\]
It follows that we can replace $r_k'$ with $r'$, as we replaced $R_k$ with $R$: for all $k \in \N$,
\[
|u(Rx) - r' \cdot R' x' - U(x)| \leq C \epsilon h_k^{1+\frac{\lambda}{2}} \text{ for all } x \in (R^{-1}X) \cap \mathcal{C}_{h_k}.
\]

Finally, we claim that $(R')^tr' = 0$, which concludes the proof.
Indeed, the inequality above tells us that the function $u(Rx) - r' \cdot R' x' - U(x)$ vanishes up to and including first order at the origin.
From \eqref{eqn: tanget at 0 is 0}, we know that $\nabla u(0) = 0 = \nabla U(0)$, forcing $(R')^tr' = 0$.
\end{proof}

%~~~~~~~~~~~~~~~~~~~~~~~~~~~~~~~~~~~~~~~~~~~~~~~~~~~~~~~%
\section{Proof of Theorem~\ref{thm: global C2 in plane}}
\label{sec: 2D}

The proof of Theorem~\ref{thm: global C2 in plane} has three steps. 
First, we prove a strict obliqueness estimate.
This key estimate allows us to find an affine transformation that aligns $\nu_{\partial X}(0)$ and $\nu_{\partial Y}(0)$, assuming $\nabla u(0) = 0$.
Second, after a rotation which prescribes the now aligned normals at the origin, we blow-up to the global flat setting of Section~\ref{sec: The Flat Case}.
Finally, we apply Proposition~\ref{prop: C2lambda pointwise} to find a pointwise expansion of $u$ at $0$.
Since, in this procedure, the origin was fixed arbitrarily, Theorem~\ref{thm: global C2 in plane} follows. 

\subsection{Strict Obliqueness}
\label{sec: Strict Obliqueness}
In this section, we prove our strict obliqueness estimate.

\begin{lemma}
\label{lem: strict obliqueness}
Let $X$ and $Y$ be open, bounded, and $C^1$ convex sets in $\R^2$. 
Suppose that $T_{min} = \nabla u$ is the optimal transport taking $f = ad_{\partial X}^\alpha$ to $g = bd_{\partial Y}^\beta$, where $a$ and $b$ are functions bounded away from zero and infinity in $X$ and $Y$ respectively, and $\max\{\alpha, \beta\} > 0$.
Then
\[
\nu_{\partial X}(x) \cdot \nu_{\partial Y}(\nabla u(x)) \geq \theta > 0\text{ for all } x \in \partial X,
\]
where $\theta$ depends only on the inner and outer diameters of $X$ and $Y$, $\alpha$, $\beta$, and the upper and lower bounds of $a$ and $b$.
\end{lemma}

The proof of this estimate follows the proof of the same estimate in the work of Savin and Yu (\cite[Section 3]{SY}).
We show that orthogonality (as opposed to strict obliqueness) is at odds with the volume estimate for sections.

\begin{proof}
By an approximation argument, we may assume that $a$ and $b$ are $C^1$ and that $X$ and $Y$ are $C^2$ and uniformly convex (cf. \cite{SY}).

Let us assume, without loss of generality, that $0 \in \partial X$, $\{ x_2 = 0 \}$ is tangent to $X$ (at $0$), $X \subset \{ x_2 > 0 \}$, $u(0) = 0$, and $\nabla u(0) = 0$. 
Now suppose, to the contrary, that we have orthogonality instead of strict obliqueness; then $\{ y_1 = 0\}$ is tangent to $Y$ (at $0 = \nabla u(0)$), and, without loss of generality, $Y \subset \{ y_1 > 0 \}$.
Set
\[
\Omega := \{ x \in X : x_2 < d_0 \} \cap \nabla v(\{ y \in Y : y_2 > 0 \}).
\]
(Recall $v$ is the minimal convex extension outside of $Y$ of the Legendre transform of $u$.)
Also, define
\[
\psi := \frac{\diam(Y)}{d_0}x_2.
\]

Notice that
\begin{equation}
\label{eqn: dirichlet}
u_2 \leq \psi \text{ on }\partial \Omega \cap X.
\end{equation}
Moreover, if $u_{12} = u_{21}$ along $\partial X$, then
\begin{equation}
\label{eqn: hopf}
u_{21} \geq 0 \text{ along } \partial X \cap \overline{\Omega}.
\end{equation}
Indeed, first, since $Y \subset \{ y_1 > 0 \}$ and $Y$ is tangent of the positive $y_2$-axis, the image under $\nabla u$ moves to the left as we move along $\partial X$ from the left toward the origin; if $\Gamma_X$ determines $\partial X$ near the origin,
\[
u_1(x_1, \Gamma_X(x_1)) \geq u_1(z_1, \Gamma_X(z_1)) \text{ if } x_1 < z_1 \leq 0.
\]
As $u_1(z_1,\Gamma_X(x_1)) \geq u_1(x_1,\Gamma_X(x_1))$, by the convexity of $u$, we deduce that
\[
u_1(z_1,\Gamma_X(x_1)) \geq u_1(z_1, \Gamma_X(z_1)). 
\]
Finally, since $\Gamma_X(x_1) > \Gamma_X(z_1)$, it follows that
\[
u_{21}(z_1,\Gamma_X(z_1)) = u_{12}(z_1,\Gamma_X(z_1))\geq 0,
\]
letting $x_1$ tend to $z_1$. 
Suppose the maximum of $u_2 - \psi$ is achieved at some $z \in \Omega$.
Then $\nabla u_2(z) = \nabla \psi(z)$.
And setting
\[
L w := u^{ij}\partial_{ij}w,
\]
we find that, at $z$,
\[
\begin{split}
0 \geq L (u_2 - \psi) &= \alpha \frac{\nu_{\partial X_d} \cdot \e_2}{d_{\partial X}} + \frac{a_2}{a} - \frac{\diam(Y)}{d_0}\bigg(\beta \frac{\nu_{\partial Y_d}(\nabla u) \cdot \e_2}{d_{\partial Y}(\nabla u) } + \frac{b_2(\nabla u) }{b(\nabla u)}\bigg)
\\
&\geq \alpha \frac{\nu_{\partial X_d} \cdot \e_2}{d_{\partial X}}- C_a - \frac{\diam(Y)}{d_0}\bigg(\beta \frac{\nu_{\partial Y_d}(\nabla u) \cdot \e_2}{d_{\partial Y}(\nabla u) } +  C_b\bigg) 
\\
&> 0.
\end{split}
\]
Indeed, $\nu_{\partial X_d} \cdot \e_2 > 0$, if $d_0 \ll 1$.
Moreover, $\nu_{\partial Y_d}(\nabla u) \cdot \e_2 < 0$, as $\nabla u(\Omega) \subset \{y_2 > 0 \}$ and, also, provided that $d_0 \ll 1$.
Finally, if $d_0 \ll 1$, then the terms with $d_{\partial \ast }$ in the denominator will be large enough to absorb the constants $C_a$ and $C_b$.
(Recall $\max\{\alpha, \beta\} > 0$.)
Here, for example,
\[
\partial X_d := \{ x \in  X : d_{\partial X}(x) = d \}
\]
and $\nu_{\partial X_d}$ is the unit normal to $\partial X_d$ oriented to point inside $\{ x \in X : d_{\partial X}(x) > d \}$.
(See, e.g., \cite{Giusti}.)
This is a contradiction.
In turn, by \eqref{eqn: dirichlet} and \eqref{eqn: hopf}, the maximum of $u_2 - \psi$ is achieved on $\partial \Omega \cap X$, which implies that
\begin{equation}
\label{eqn: max p s.o.}
u \leq \frac{\diam(Y)}{d_0} x_2^2 \text{ in } \Omega.
\end{equation}

Unfortunately, we cannot guarantee that $u_{12} = u_{21}$ along $\partial X$.
Therefore, we consider the following approximation scheme: let $\nabla u^k$ be the solution to the optimal transport problem taking 
\[
f^k := (1-k^{-1})f + k^{-1}\|f\|_{L^1(X)} \text{ to }
g^k := (1-k^{-1})g + k^{-1}\|g\|_{L^1(Y)}.
\]
By \cite[Remark 2.1]{SY}, $\nabla u^k$ converges to $\nabla u$ locally uniformly in $\R^2$, and, similarly, $\nabla v^k$ converges to $\nabla v$ locally uniformly in $\R^2$.
(Here $v^k$ is the minimal convex potential associated to the optimal transport problem taking $g^k$ to $f^k$, and $v^k$, in $Y$, agrees with the Legendre transform of $u^k$.)
Set 
\[
\Omega_k := \{ x \in X : x_2 < d_0 \} \cap \nabla v^k(\{ y \in Y : y_2 > 0 \}).
\]
Then $\overline{\Omega}_k$ converges to $\overline{\Omega}$ in the Hausdorff sense.
Since $f_k$ and $g_k$ are positive and H\"older continuous, $u^k \in C^2(\overline{X})$ by Caffarelli's boundary regularity theory (\cite{C3}).
So $u^k_{21} = u^k_{12}$ along $\partial X$.
In turn, by the formal maximum principle argument above,
\[
u^k \leq \frac{\diam(Y)}{d_0} x_2^2 \text{ in } \Omega_k.
\]
Taking the limit, $k \to \infty$, proves \eqref{eqn: max p s.o.}.

In summary, if we have orthogonality rather than strict obliqueness,
\begin{equation}
\label{eqn: strict obliqueness upper barriers}
u \leq Cx_2^2 \text{ in $\Omega(u)$ and } v \leq Cy_1^2 \text{ in $\Omega(v)$},
\end{equation}
where the estimate on $v$ is by duality.

To conclude, let $\Gamma_Y$ determine $\partial Y$ near the origin.
Corollaries~\ref{cor: volume product} and \ref{cor: image of section is comp to dual section} and \eqref{eqn: strict obliqueness upper barriers} applied in succession imply 
\[
\begin{split}
Ch^2 &\geq |S_h(u,0)||\nabla u(S_h(u,0))| 
\\
&\geq c|S_h(u,0)||S_h(v,0)|  
\\
&\geq c|\{ x \in \overline{X} : x_1 \leq 0, x_2 \leq c h^{\frac{1}{2}} \}||\{ y \in \overline{Y} : y_2 \leq 0, y_1 \leq c h^{\frac{1}{2}} \}|
\\
&
\geq ch^{\frac{1}{2}}\Gamma_X^{-1}(ch^{\frac{1}{2}})h^{\frac{1}{2}}\Gamma_Y^{-1}(ch^{\frac{1}{2}}).
\end{split}
\]
Dividing through by $h^2$ yields
\[
C \geq \frac{ \Gamma_X^{-1}(ch^{\frac{1}{2}}) }{ch^{\frac{1}{2}}} \frac{ \Gamma_Y^{-1}(ch^{\frac{1}{2}}) }{ch^{\frac{1}{2}}}.
\]
Since $\Gamma_X'(0) = \Gamma_Y'(0) = 0$, by assumption, $(\Gamma_X^{-1})'(0) = (\Gamma_Y^{-1})'(0) = + \infty$.
But this implies that the right-hand side above tends to infinity when $h$ tends to zero, which is impossible.
\end{proof}

\subsection{Blow-ups}
\label{sec: blow-ups}

In this section, we blow-up.
That said, in order to blow-up to the flat setting studied in Section~\ref{sec: The Flat Case}, we have to not only use Lemma~\ref{lem: strict obliqueness} but choose the right transformation to normalize sections.

\subsubsection{A First Normalization}

Up to a translation and subtracting an affine function, we assume that
\[
0 \in \partial X \cap \partial Y,\, u(0) = 0 = v(0),\, \text{and } \nabla u(0) = 0 = \nabla v(0).
\] 
From our strict obliqueness estimate, a shearing transformation exists that aligns the inner unit normals of $\partial X$ and $\partial Y$ at the origin, which after a rotation can be prescribed.
In particular, a $\Theta$ exists such that
\[
\tilde{X} := \Theta^{-1}X \subset \{ x_n > 0\} \text{ and } \tilde{Y} := \Theta^tY \subset \{ y_n > 0\}
\]
have $\{ x_2 = 0 \}$ and $\{ y_2 = 0\}$ as tangent planes to their boundaries at $0$ respectively.
Moreover, $\det \Theta = 1$.
Then defining
\[
\tilde{u}(x) := u(\Theta x),
\]
we find that
\[
(\nabla \tilde{u})_{\#} \tilde{f} = \tilde{g}
\]
where
\[
\tilde{f} = \tilde{a} d^\alpha_{\partial \tilde{X}} \text{ with } \tilde{a}(x) := a(\Theta x) \bigg[\frac{d_{\partial X}(\Theta x)}{d_{\partial \tilde{X}}(x)} \bigg]^\alpha
\]
and
\[
\tilde{g} = \tilde{b} d^\beta_{\partial \tilde{Y}} \text{ with } \tilde{b}(y) := b(\Theta^{-t} y) \bigg[\frac{d_{\partial Y}(\Theta^{-t} y)}{d_{\partial \tilde{Y}}(y)} \bigg]^\beta.
\]
By \cite[Lemma~6.1]{LS}, $\tilde{a} \in C^{\mu}(\tilde{X})$ and $\tilde{a} > 0$.
Similarly, $\tilde{b} \in C^{\omega}(\tilde{Y})$ and $\tilde{b} > 0$.

\begin{remark}
The restrictions on $\lambda$ with respect to $\alpha$ and $\beta$ explicitly, rather than via $\gamma$, come from this normalization.
Indeed, $\tilde{a}$, for instance, as the product of two H\"{o}lder continuous functions, will be H\"{o}lder continuous.
Yet between the two H\"{o}lder exponents it could inherit, it will inherit the smaller one.
\end{remark}

It will be convenient to suppress the tildes in these definitions, and write $u$ rather than $\tilde{u}$, for example.

\subsubsection{A New Ellipsoid}

In the proof of Lemma~\ref{lem: pure normal comp to 1}, we found ellipsoids comparable to sections whose axes were parallel to the coordinate axes.
The same construction applies here.
The only difference is that the $\delta$ in \eqref{eqn: new ellipsoid} now depends on the doubling constants of $f$ and $g$, dimension, and the Lipschitz semi-norm of $\partial X$.
For simplicity, we let $w_h = w_{1;h}$, as defined a few lines above \eqref{eqn: new ellipsoid}.
It will be convenient to normalize these ellipsoids as we blow-up.

\subsubsection{A Blow-up Limit}

In this section, we show that the boundaries of $X$ and $Y$ flatten as we normalize.
Since $X$ and $Y$ are $C^{1,1}$ and uniformly convex, near the origin,
\[
\{ x_2 > p^{-1} x_1^2 \} \subset X \subset \{ x_2 > p x_1^2 \}
\text{ and }
\{ y_2 > q^{-1} y_1^2 \} \subset Y \subset \{ y_2 > q y_1^2 \}
\]
for two constants $p, q > 0$.

\begin{lemma}
\label{lem: bdry does not flatten}
For each $k \in \N$, there is an $h_k > 0$  such that $d_{h_k} > k w_{h_k}^2$.
\end{lemma}

\begin{proof}
Suppose the lemma fails to hold; that is, a $\tilde{k} \in \N$ exists for which
\[
d_h \leq \tilde{k} w_h^2 \text{ for all } h \leq \tilde{h}.
\]
Let
\[
A_h := \diag(w_h,d_h),\, X_h := A_h^{-1}X,\, \text{and } Y_h := h^{-1}A_h^t Y,
\]
and consider the normalized potentials
\[
u_h(x) := \frac{u(A_h x)}{h}
\text{ and } 
v_h(x) := \frac{v(A_h^{-t}h y)}{h}
\]
along with their normalized densities
\[
f_h(x) := ( \det A_h )f(A_h x)
\text{ and }
g_h(y) := h^2 ( \det A_h^{-t})g(A_h^{-t}hy).
\]
Now observe that
\[
\{ x_2 > d_h^{-1}w_h^2 p^{-1} x_1^2 \} \subset X_h \subset \{ x_2 > d_h^{-1}w_h^2p x_1^2 \}
\]
and
\[
\{ y_2 > h d_h w_h^{-2} q^{-1} y_1^2 \} \subset Y_h \subset \{ y_2 > h d_h w_h^{-2}q y_1^2 \}.
\]
By convexity,
\[
d_h \geq cpw_h^2.
\]
Thus, taking $\tilde{k}$ larger if needed,
\[
\frac{1}{\tilde{k}} \leq cp \leq \frac{d_h}{w_h^2} \leq \tilde{k}.
\]
Hence, up to a subsequence, as $h \to 0$, we find two limiting domains $\tilde{X}$ and $\tilde{Y}$ such that
\[
\{ x_2 > \tilde{k}p^{-1} x_1^2 \} \subset \tilde{X} = \{ x_2 > \tilde{p}(x_1) \} \subset \{ x_2 > p \tilde{k}^{-1} x_1^2 \},
\]
for some convex $\tilde{p}$ that vanishes only at $0$, and
\[
\tilde{Y} = \{ y_2 > 0 \}.
\]
Moreover, recalling our balancing condition (\eqref{eqn: mass balance doubling}),
\[
\frac{1}{C} \leq m_h := \frac{d_h^{2 + \alpha +\beta} w_{h}^2}{h^{2+\beta}} \leq C,
\] 
we find a convex function $\tilde{u}$, smooth in $\tilde{X}$, and such that
\[
\det D^2\tilde{u}(x) = \tilde{m} \frac{a(0)(x_2 - \tilde{p}(x_1))^\alpha}{b(0)\tilde{u}_2^\beta} \text{ in } \{ x_2 > \tilde{p}(x_1) \}
\]
and
\[
\tilde{u}_2 = 0 \text{ along } \{ x_2 = \tilde{p}(x_1) \}.
\]
Up to multiplying $\tilde{u}$ by a constant, we may assume that $\tilde{m} a(0)/b(0) = 1$.
(The constant $\tilde{m} = \lim_{h \to 0} m_h$.)

Set
\[
Lw :=  \tilde{u}^{ij} \partial_{ij} w.
\]
First, notice that there is a $C > 0$ such that
\[
Cx_2^{\gamma} \geq \tilde{u}_2 \text{ on } \partial S_1(\tilde{u},0).
\] 
If $s := \sup_{S_1(\tilde{u},0)} (\tilde{u}_2-Cx_2^{\gamma})$ is achieved at $\tilde{x} \in S_1(\tilde{u},0)$, then $\tilde{u}_2(\tilde{x}) = C\tilde{x}_2^{\gamma} + s$ for some $s \geq 0$, and
\[
\begin{split}
0 &\geq L(\tilde{u}_2 - Cx_2^{\gamma})(\tilde{x})
\\
% &= (\alpha \log (x_2 - \tilde{p}(x_1))  - \beta \log \tilde{u}_2 )_2(\tilde{x}) - \tilde{u}^{22} C \gamma (\gamma - 1)\tilde{x}_{2}^{\gamma - 2} 
% \\
&= \alpha \frac{1}{\tilde{x}_{2} - \tilde{p}(\tilde{x}_1)} - \beta \frac{C\gamma \tilde{x}_{2}^{\gamma-1}}{C\tilde{x}_{2}^\gamma + s} - C^{-1}\gamma^{-1} \tilde{x}_{2}^{1-\gamma} C \gamma (\gamma - 1)\tilde{x}_{2}^{\gamma - 2}
\\
&\geq \alpha \frac{1}{\tilde{x}_{2} - \tilde{p}(\tilde{x}_{1})} - (\beta \gamma + \gamma - 1) \frac{1}{\tilde{x}_{2}}
\\
% &=\alpha \frac{1}{\tilde{x}_{2} - \tilde{p}(\tilde{x}_{1})} - \alpha\frac{1}{\tilde{x}_{2}}
% \\
& > 0,
\end{split}
\]
provided $\alpha > 0$, an impossibility.
(Because $\tilde{u}_2$ and $Cx_2^\gamma + s$ touch at $\tilde{x}$, their gradients agree at $\tilde{x}$.
Since $\tilde{u}^{ij}$ is the inverse of $\tilde{u}_{ij}$, we have that $\tilde{u}_{12} \tilde{u}^{12}  + \tilde{u}^{22}\tilde{u}_{22} = 1$.
Moreover, $\tilde{u}_{12}(\tilde{x}) = (Cx_2^\gamma +s)_1(\tilde{x}) = 0$.
In turn, $\tilde{u}^{22}(\tilde{x})\tilde{u}_{22}(\tilde{x}) = 1$, which explains the second equality line.)
So we find that
\[
\tilde{u} \leq Cx_2^{1+\gamma} \text{ in } S_1(\tilde{u},0).
\]
If $\alpha = 0$, then consider the power $\gamma-\epsilon$ rather than $\gamma$, with $\epsilon > 0$ arbitrary but small.
In particular, $Cx_2^{\gamma-\epsilon}$, for some $C > 0$ independent of $\epsilon$, is an upper barrier; the right-hand side, in this case, becomes $\epsilon(1+\beta)\tilde{x}_2^{-1} > 0$.
Applying the maximum principle and then sending $\epsilon$ to zero, yields the same inequality.
Thus,
\[
d_t(\tilde{u},0) \geq ct^{\frac{1}{1+\gamma}}.
\]

Since $u_h$ converges to $\tilde{u}$ locally uniformly,
\[
d_t(u_h,0) \geq \frac{c}{2}t^{\frac{1}{1+\gamma}} \text{ for all } h \ll 1.
\]
Moreover,
\[
d_t(u_h,0) = \frac{d_{th}(u)}{d_h(u)} \text{ and } w_t(u_h,0) = \frac{w_{th}(u)}{w_h(u)}.
\]
So our balancing condition holds for $u_h$ as well after replacing $C$ by $C^2$.
In turn,
\[
w_t(u_h,0)^2 \leq Ct.
\]
Then
\[
c \frac{ t^{\frac{1}{1+\gamma}} }{t} 
\leq \frac{d_t(u_h,0)}{w_t(u_h,0)^2} 
= \frac{d_{th}}{d_h} \frac{w_h^2}{w_{th}^2} 
\leq \tilde{k}^2.
\]
But taking $t \ll 1$, we find this inequality impossible.
\end{proof}

Swapping the roles of $u$ and $v$ and $\alpha$ and $\beta$ (also $p$ and $q$), we find a dual lemma.

\begin{lemma}
\label{lem: dual bdry does not flatten}
For each $k \in \N$, there is an $h_k > 0$ such that $k h_k d_{h_k} < w_{h_k}^2$.
\end{lemma}

A consequence of Lemmas~\ref{lem: bdry does not flatten} and \ref{lem: dual bdry does not flatten} is that the boundaries of $X$ and $Y$ flatten under $A_h^{-1}$ and $hA_h^t$ respectively; up to a subsequence, in the Hausdorff sense,
\[
X_h \to \{ x_2 > 0 \} \text{ and } Y_h \to \{ y_2 > 0 \},
\]
as $h \to 0$.
In turn, up to multiplication by a constant and the same subsequence, locally uniformly in $\R^2$,
\[
u_h \to \bar{u} = \bar{P}(x_1) + \bar{p}\frac{\gamma^{\frac{\beta}{1+\beta}} }{(1+\gamma)\gamma} x_2^{1+\gamma},
\]
by Theorem~\ref{thm: liouville}.
Furthermore,
\[
|\nabla \bar{u}| \leq \frac{1}{r} \text{ in } S_1(\bar{u},0),
\]
where $r > 0$ is the constant from Corollary~\ref{cor: sections are ellipsoids}.

\subsection{Conclusion}
Up to a multiplication by a constant and determinant $1$ transformation in the $x'$ variables (both depending on the constant $r > 0$ from Corollary~\ref{cor: sections are ellipsoids}, i.e., only on the doubling constants of $f$ and $g$), we find that, after choosing $h > 0$ sufficiently small, the rescaling
\[
\tilde{u} := u_h
\]
satisfies the hypothesis of Proposition~\ref{prop: C2lambda pointwise}.
First, notice that 
\[
\tilde{u}(0) = 0 = |\nabla \tilde{u}(0)|,
\]
which is \eqref{eqn: tanget at 0 is 0}, since $u(0) = 0 = |\nabla u(0)|$.
Second, for any $\epsilon > 0$,
\[
|\tilde{u} - U| \leq \epsilon  \text{ in } \mathcal{C}_2,
\]
by the conclusion of the previous section.
(The constant and determinant $1$ transformation turn $\bar{u}$ into $U$.)
By convexity,
\[
\nabla \tilde{u}(\tilde{X} \cap \mathcal{C}_1) \subset \tilde{Y} \cap \mathcal{C}_{1/\rho}^\ast.
\]
And, by construction, $\nabla \tilde{u}$ maps $\partial \tilde{X}$ to $\partial \tilde{Y}$.
Also,
\[
\{ x_2 > d_h^{-1}w_h^2t^{\frac{\gamma}{1+\gamma}}p^{-1}  x_1^2  \}\subset \tilde{X} \subset \{ x_2 > 0 \}
\]
and
\[
\{ y_2 > h d_h w_h^{-2} t^{\frac{1}{1+\gamma}} q^{-1} y_1^2 \} \subset \tilde{Y} \subset \{ y_2 > 0 \}.
\]
So \eqref{eqn: bdrys geometrically flat} follows by Lemmas~\ref{lem: bdry does not flatten} and \ref{lem: dual bdry does not flatten} as well as the definition of $\lambda$.
These estimates on $\partial \tilde{X}$ and $\partial \tilde{Y}$ together with the estimates
\[
|\tilde{a}(x) - 1| \leq \delta \epsilon|x|^\mu \text{ and } |\tilde{b}(y) - 1| \leq \delta \epsilon|y|^\omega
\]
imply that the inequalities on $\det D^2 \tilde{u}$ hold.
Thus, applying Proposition~\ref{prop: C2lambda pointwise} proves that $u$ is $C^{2+\lambda}$ at $0$ when $\gamma \geq 1$ and $C^{1+\gamma(1+\omega)}$ at $0$ when $\gamma < 1$.

Since $\Theta$ and $r > 0$ are uniform over points in $\partial X$, Theorem~\ref{thm: global C2 in plane} is proved.

%~~~THE BIBLIOGRAPHY~~~%

\end{document}